\documentclass[11pt,reqno]{amsart}
\usepackage{CJK}
\usepackage{hyperref}
\usepackage{cite}
\usepackage{amsmath}
\usepackage{mathrsfs}

\makeatletter
\@namedef{subjclassname@2020}{%
	\textup{2020} Mathematics Subject Classification}
\makeatother

\numberwithin{equation}{section}

\newtheorem{theorem}{Theorem}[section]
\newtheorem{lemma}[theorem]{Lemma}
\newtheorem{definition}[theorem]{Definition}

\newtheorem{proposition}[theorem]{Proposition}
\newtheorem{remark}[theorem]{Remark}

\allowdisplaybreaks

\begin{document}
\title[\hfil Regularity for degenerate fully nonlinear nonlocal equations] {Regularity theory for degenerate fully nonlinear nonlocal equations with a Hamiltonian term}
\author[Y. Fang, J. Kinnunen, C. Zhang]{Yuzhou Fang, Juha Kinnunen and Chao Zhang}


\address{Yuzhou Fang\hfill\break School of Mathematics, Harbin Institute of Technology, Harbin 150001, China}
 \email{18b912036@hit.edu.cn}

\address{Juha Kinnunen \hfill\break Department of Mathematics, P.O. Box 11100, Aalto University, Finland}
\email{juha.k.kinnunen@aalto.fi}

\address{Chao Zhang\hfill\break School of Mathematics and Institute for Advanced Study in Mathematics, Harbin Institute of Technology, Harbin 150001, China}
 \email{czhangmath@hit.edu.cn}

\subjclass[2020]{35B65; 35D40; 35J70; 35R11}   \keywords{Regularity; viscosity solution; nonlocal fully nonlinear degenerate equation; Hamiltonian terms}

\begin{abstract}
We investigate a class of degenerate fully nonlinear nonlocal elliptic equations with Hamiltonian terms. By precisely characterizing the interaction between the degeneracy law of equations and the growth behavior of the Hamiltonian terms, we establish the Lipschitz regularity of viscosity solutions by the Ishii-Lions method, and further show the gradient H\"{o}lder continuity for solutions via utilizing perturbation techniques. Additionally, under minimal assumptions on the degeneracy pattern, the $C^1$-differentiability property of solutions is explored as well. 
\end{abstract}

\maketitle

\section{Introduction}
\label{sec0}

This work discusses the gradient regularity of viscosity solutions to a class of degenerate fully nonlinear integro-differential equations with Hamiltonian terms
\begin{equation}
\label{main}
-\Phi(x,Du)\mathcal{I}_\sigma(u,x)+H(x,Du)=f(x) \quad\text{in }  B_1,
\end{equation}
where $\sigma\in (0,2), \Phi(x,0)=0, f\in C(B_1)\cap L^\infty(B_1)$ and the conditions on the law of degeneracy $\Phi$ are given in Section \ref{sec2}, and the fully nonlinear nonlocal operator $\mathcal{I}_\sigma$ is uniformly elliptic in the sense of Caffarelli and Silvestre \cite{CS09,CS11}, i.e.,
$$
\inf_{I\in\mathcal{L}}Iv(x)\le \mathcal{I}_\sigma(u+v,x)-\mathcal{I}_\sigma (u,x)\le\sup_{I\in\mathcal{L}}Iv(x)
$$
for a set of linear operators $\mathcal{L}$. Here $B_1=B_1(0)$ is the unit ball in the Euclidean space $\mathbb{R}^N$, and the Hamiltonian term $H(x,Du)$ fulfills some appropriate preconditions stated in Section \ref{sec2}. Let $\mathcal{K}$ be a collection of symmetric kernels consisting of measurable functions $K:\mathbb{R}^N\setminus\{0\}\rightarrow\mathbb{R}^+$ satisfying
$$
\lambda\frac{C_{N,\sigma}}{|x|^{N+\sigma}}\le K(x)\le \Lambda\frac{C_{N,\sigma}}{|x|^{N+\sigma}},
$$
where $0<\lambda\le\Lambda<\infty$ and $C_{N,\sigma}>0$ is a normalizing constant. 
For $K\in \mathcal{K}$ and $u:\mathbb{R}^N\rightarrow\mathbb{R}$, let
$$
I_Ku(x)=\frac{1}{2}\mathrm{P.V.}\int_{\mathbb{R}^N}(u(x+y)+u(x-y)-2u(x))K(y)\,dy
$$
with the symbol P.V. representing the Cauchy principal value. Notice that $I_Ku$ for each $K$ is well-defined, provided $u$ is $C^{1,1}$ in a neighborhood the point $x$ and fulfills an adequate growth condition at infinity
$$
\|u\|_{L^1_\sigma(\mathbb{R}^N)}=\int_{\mathbb{R}^N}\frac{|u(y)|}{1+|y|^{N+\sigma}}\,dy<\infty.
$$
At this point, we denote as $L^1_\sigma(\mathbb{R}^N)$ the set of such functions. For a two-parameter family of kernels $\{K_{\alpha\beta}\}_{\alpha\beta}\subseteq\mathcal{K}$,
a nonlinear operator
$$
\mathcal{I}_\sigma (u,x)=\inf_\beta\sup_\alpha I_{K_{\alpha\beta}}u(x)
$$
naturally arises from stochastic control problems, see \cite{Soner}. Namely in competitive stochastic games involving two or more players, they are allowed to choose diverse strategies at every step to maximize the desired value of some functions at the first exit point of domain. 

Over the past years, fully nonlinear integro-differential equations have been receiving increasing attention, starting with the pioneering works \cite{CS09,CS11}. In particular, Caffarelli and Silvestre \cite{CS09} established a series of interior behaviours for viscosity solutions to such equations, including the Aleksandrov--Bakelman--Pucci estimate and Harnack's inequality, as well as $C^{0,\alpha}$ and $C^{1,\alpha}$ estimates. Moreover, these results remain stable as the degree of operator tends to two, thus which could be regarded as a natural generalization of the regularity theory for elliptic PDEs. Subsequently, the $C^{1,\alpha}$-regularity was extended, by the same authors \cite{CS11}, to nonlocal equations that are not necessarily translation-invariant by applying compactness and perturbative techniques. Their argument was based on the observation that solutions to the considered equation are $C^{1,\alpha}$ regular, provided the equation is uniformly close to another one having $C^{1,\alpha}$ solutions. In addition, for a wide class of integro-differential equations encompassing first and second-order terms, one can refer to \cite{BCI11, BCCI12} concerning H\"{o}lder or Lipschitz continuity by developing the Ishii--Lions viscosity method for nonlocal versions.

Recently, dos Prazeres and Topp \cite{PT21} first studied the fully nonlinear fractional equations
$$
-|Du|^p\mathcal{I}_\sigma (u, x)=f(x)
$$
that degenerate with the gradient and proved the interior H\"{o}lder regularity by means of the ideas in \cite{BCI11, BCCI12}, as well as the gradient H\"{o}lder estimate via an improvement of flatness as $\sigma$ sufficiently close to 2. For related results on such equations, we also refer to \cite{BKLT15, APT23}. Subsequently, these regularity properties were generalized to the nonlinear integro-differential equations with nonhomogeneous degeneracy of the type
$$
-(|Du|^p+a(x)|Du|^q)\mathcal{I}_\sigma (u, x)=f(x),   \quad a(x)\ge0.
$$
In particular, the authors in \cite{APS24} inferred that, for any $1<\sigma<2$, there exists at least one $C^{1, \alpha}$-regular viscosity solution in the case where $0<p\le q$. Fang, R\u{a}dulescu and Zhang \cite{FRZ25} later examined nonlocal problems with more general structures, and verified an improved gradient estimate at the critical point of solutions. Furthermore, the borderline regularity for the aforementioned nonhomogeneous equations was established in \cite{WJ25}. Moreover, we refer to \cite{OT23} for the $C^{1, \alpha}$ regularity on variable-exponent degenerate mixed local and nonlocal equations.

On the other hand,  when $\Phi(x, Du)\equiv1$ and $\mathcal{I}_\sigma (u,x)=(-\Delta)^{\frac{\sigma}{2}}u(x)$,  the equation in \eqref{main} reduces to a class of nonlocal Hamilton--Jacobi equations of the form
$$
(-\Delta)^{\frac{\sigma}{2}}u(x)+H(x,Du)=f(x).
$$
Some research has been conducted on such equations, with results tailored to specific requirements on the Hamiltonian $H$. For instance, Barles, Koike, Ley and Topp \cite{BKLT15} concluded H\"{o}lder regularity for bounded viscosity solutions when the coercive gradient term has the stronger effect, which was applied to derive the ergodic asymptotic behaviour of the related parabolic problem; see \cite{BT16} for Lipschitz continuity of censored subdiffusive integro-differential equations, and \cite{CR11} for a probabilistic approach. For unbounded viscosity solutions to the previous Hamilton--Jacobi equation (with an Ornstein--Uhlenbeck drift), Chasseigne, Ley and Nguyen \cite{CLN19} studied the Lipschitz estimate under the scenario that the Hamiltonian is sublinear. When the fractional Laplacian is replaced by a fully nonlinear nonlocal operator, i.e., $\mathcal{I}_\sigma (u, x)+H(x, Du)=f(x)$, the existence of boundary blow-up solutions to associated nonlocal Dirichlet problems was verified by \cite{DQT24}. More recently, for fully nonlinear nonlocal Hamilton--Jacobi equations, Liouville theorems on solutions were considered in \cite{BQT25} by applying the Ishii--Lions type technique. We also note that in the special case where the Hamiltonian is a drift term (i.e., $H(x, Du)=b(x)\cdot Du$), Quaas, Salort and Xia \cite{QSX20} discussed the principal eigenvalues of such equations. More related results can be found in e.g. \cite{BK23, BT25a, DJZ18, CD12} and references therein.

To the best of our knowledge, higher order regularity results, particularly on $C^{1, \alpha}$ regularity, was still unknown for the above-mentioned fully nonlinear nonlocal Hamilton-Jacobi problems, let alone for those with degeneracy as \eqref{main}. On the other hand, when $\Phi(x, Du)\equiv |Du|^p$, the investigation of Eq. \eqref{main} is also motivated by its local counterpart, i.e., second-order fully nonlinear equations of the form
\begin{equation}
\label{1-2}
-|Du|^pF(D^2u)+H(x,Du)=f(x) \quad \text{in } B_1,
\end{equation}
where the operator $F: \mathcal{S}^N\rightarrow\mathbb{R}$ is uniformly elliptic in the sense that
$$
\lambda\mathrm{Tr}(B)\le F(A+B)-F(A)\le\Lambda\mathrm{Tr}(B)
$$
for all $A, B\in\mathcal{S}^N$, $B\ge0$ and $\mathcal{S}^N$ is the set of symmetric matrices. For the situation that $H(x, Du)\equiv0$, various aspects of this kind of equations have been already examined: properties of eigenfunctions and eigenvalues \cite{BD10}, ABP estimate and Harnack inequality \cite{Imb11}, $C^{1,\alpha}$-regularity issues \cite{ART15,IS13}, Schauder-type theory \cite{Nas24}, $W^{2,\delta}$-type estimate \cite{BKO25}, more general degeneracy \cite{BBLL,DeF,FRZ21,JLMS} and so on. When it comes to the gradient H\"{o}lder continuity under the more general scenario where $H(x,Du)=b(x)|Du|^q$, we refer readers to \cite{BD16} for the case $0\le q\le p+1$ and \cite{BDL19} for the case $p+1<q\le p+2$. We also mention that for $0\le q\le p+1$, Andrade and Nascimento \cite{AN25} proved the sharp regularity estimate for viscosity solutions.

Motivated by the aforementioned works, the purpose of this paper is twofold: first, to develop a viscosity solution theory for the nonlocal equation \eqref{main}, corresponding to that established for degenerate fully nonlinear elliptic problems (e.g., \eqref{1-2}); second, to explore the higher order regularity theory for fully nonlinear nonlocal Hamilton-Jacobi equations with degeneracy. More precisely, we aim to find suitable structural assumptions on \eqref{main} that enable us to establish its interior $C^{0,1}$, $C^1$ and $C^{1, \alpha}$ regularity in a universal way. A key challenge arises from the coupled interplay between the Hamiltonian term and the potentially degenerate gradient term, which introduces non-trivial technical difficulties. To derive gradient continuity results, we assume that the fractional index $\sigma$ in \eqref{main} is close enough to two. This ensures the nonlinear nonlocal operator $\mathcal{I}_\sigma$ approximates a uniformly elliptic operator $F$, allowing the regularity of solutions to \eqref{main} to be inherited, in an appropriate manner, from that of $F$-harmonic functions. Here, $F$-harmonic functions refer to the viscosity solutions of the equation $F(D^2 u)=0$.

To simplify the discussion, we set $\Phi(x, Du) \equiv |Du|^p$ in Theorems \ref{thm1} and \ref{thm2} below. With this notation, we are in a position to state the $C^{1, \alpha}$ regularity result for scenarios where the Hamiltonian's growth does not exceed \(p+1\).

\begin{theorem}
\label{thm1}
Let $u\in C(\overline{B}_1)$ be a viscosity solution to \eqref{main}, and assume that the conditions $(A_1)$--$(A_4)$ in Section \ref{sec2} hold with $0\le q\le p+1$.
Then there exists $\sigma_0\in (1, 2)$, close enough to $2$, such that if $\sigma\in(\sigma_0, 2)$ then $u$ is locally $C^{1,\alpha}(B_1)$-regular with the following estimates:
\begin{itemize}
	\item [(i)] if $q< p+1$, then
	$$
	\|u\|_{C^{1,\alpha}(B_{\frac12})}\le C\left(\|u\|_{L^\infty(B_{1})}+\|u\|_{L^1_\sigma(\mathbb{R}^N)}+\mathcal{H}^\frac{\sigma-1}{p+1-q}+(\mathcal{M}+\|f\|_{L^\infty(B_{1})})
	^\frac{\sigma-1}{1+p}\right),
	$$
	where the constant $C\ge1$ depends on $N$, $\lambda$, $\Lambda$, $p$, $q$ and $\alpha$;
	
	\item [(ii)] if $q=p+1$, then
	$$
	\|u\|_{C^{1,\alpha}(B_{\frac12})}\le \overline{C}\bigl(\|u\|_{L^\infty(B_{1})}+\|u\|_{L^1_\sigma(\mathbb{R}^N)}\bigr),
	$$
	where the constant $\overline{C}$ depends in addition on $\mathcal{M}$, $\mathcal{H}$ and $\|f\|_{L^\infty(B_{1})}$.
\end{itemize}
Here $0<\alpha<\min\bigl\{\overline{\alpha},\frac{\sigma-1}{1+p}\bigr\}$
and $\overline{\alpha}$ is the exponent corresponding to the optimal regularity for an $F$-harmonic function.
\end{theorem}

\begin{remark}
The constant $C$ or $\overline{C}$ in Theorem \ref{thm1} is uniform in $\sigma$, that is, it does not blow up as $\sigma\rightarrow2$. The index of gradient H\"{o}lder continuity stated above is also sharp in the spirit of \cite{AN25, ART15, PT21}.
\end{remark}

Next, we present the Lipschitz continuity of viscosity solutions to \eqref{main} under the condition that the growth order $q$ of the Hamiltonian $H$ is between $p+1$ and $p+\sigma$.
We would like to mention that the condition $q\le p+\sigma$ mirrors the constraint $q\le p+2$ for the fully nonlinear (local) equations with the Hamiltonian term 
due to Birindelli, Demengel and Leoni\cite{BDL19,BD16}. Since our Lipschitz results are stable as $\sigma$ approaches to two, the latter condition is a natural limit case of the former. We will clarify in detail the derivation of the condition $q\le p+\sigma$ at the end of Section \ref{sec3}.

\begin{theorem}
\label{thm2}
Let $\sigma\in(1,2)$ and assume that the hypotheses $(A_1)$--$(A_3)$ in Section \ref{sec2} hold with $p+1< q\le p+\sigma$.
Suppose that $u\in C(\overline{B}_1)$ is a viscosity solution to \eqref{main}. Then $u$ is locally Lipschitz continuous in $B_1$.
More precisely, there exists a universal constant $C\ge1$ such that
$$
|u(x)-u(y)|\le C|x-y|
$$
for every $x,y\in B_{\frac12}$, where the constant $C$ is uniformly bounded as $\sigma\rightarrow2$.
\end{theorem}

\begin{remark}
Although the two theorems above are derived for the case $\Phi(x, Du)\equiv|Du|^p$, the statements are indeed also valid for the general scenario where $\Phi(x,Du)\sim|Du|^p$ and there are barely any differences for the proof of conclusions.
\end{remark}

We are now ready to give the third regularity result by relaxing the degeneracy conditions of the equation. For the fully nonlinear equation
$$
\gamma(|Du|)F(D^2u)=f(x),
$$
if no conditions are imposed on $\gamma$ whatsoever, solutions to this equation may not be differentiable. Particularly, Andrade, Pellegrino, Pimentel and Teixeira \cite{APPT22} proposed minimal hypotheses on the modulus of continuity $\gamma$ to ensure solutions preserve the $C^1$-differentiability. Here, for the nonlocal model \eqref{main}, we intend to establish the similar borderline regularity theory as well, which is new even for the local analogue of \eqref{main}.

\begin{theorem}
\label{thm3}
Let $u\in C(B_1)$ be a viscosity solution to \eqref{main} with $\Phi(x,Du)\equiv \gamma(|Du|)$, and assume that the conditions $(A_1)$, $(A_2)$, $(A_4)$--$(A_6)$ in Section \ref{sec2} hold. Then there exists $\sigma_0\in(1,2)$, sufficiently close to $2$, such that if $\sigma\in(\sigma_0, 2)$, then $u\in C^1_{\rm loc}(B_1)$. More precisely, there exists a modulus of continuity $\omega:\mathbb{R}^+_0\rightarrow\mathbb{R}^+_0$, which depends only on $N$, $\sigma$, $\gamma$, $\lambda$, $\Lambda$, $\|u\|_{L^\infty(B_1)}$ and $\|f\|_{L^\infty(B_1)}$, such that
$$
|Du(x)-Du(y)|\le \omega(|x-y|)
$$
for every $x,y\in B_{\frac14}$.
\end{theorem}

This paper is organized as follows. In Section \ref{sec2}, we first collect some basic notations, notions of viscosity solution, and give the assumptions on \eqref{main} together with a necessary auxiliary lemma. Section \ref{sec3} is dedicated to showing the $C^{0,1}$ property of solutions, and further we in Section \ref{sec4} demonstrate the $C^{1,\alpha}$-regularity for \eqref{main} based on the compactness result and the improvement of flatness idea. Finally, the $C^1$-differentiability conclusion of solutions to \eqref{main} is justified in Section \ref{sec5}.

\section{Preliminaries}
\label{sec2}

We in this part present some structure conditions on \eqref{main}, and give the definition of viscosity solution as well as a useful auxiliary result. Throughout this manuscript, we always assume the nonlocal operator $\mathcal{I}_\sigma$ is uniformly elliptic. In addition, we assume that the following following conditions hold:
\begin{itemize}
\item [($A_1$)] The viscosity solution $u$ is of $L^1_\sigma(\mathbb{R}^N)$ for the fractional operator is well-defined on the complement of some bounded domain;
\item [($A_2$)] The source term $f$ belongs to $C(B_1)\cap L^\infty(B_1)$;
\item [($A_3$)] In the scenario $\Phi(x,Du)=|Du|^p$ with $0\le p<\infty$, we suppose that the Hamiltonian term $H:B_1\times \mathbb{R}^N\rightarrow \mathbb{R}$ is continuous
and that there exist two constants $\mathcal{M},\mathcal{H}$ such that
    \begin{equation}
    \label{2-1}
    |H(x,\xi)|\le \mathcal{M}+\mathcal{H}|\xi|^q, \quad q\ge0,
    \end{equation}
    for every $x\in B_1, \xi\in \mathbb{R}^N$;
\item [($A_4$)] Let $\{K_{ij}\}_{ij}\subset \mathcal{K}$ be a collection of kernels such that there exist numbers $\{k_{ij}\}\subset[\lambda,\Lambda]$ and a modulus of continuity $\omega$ satisfying
    $$
    \left|K_{ij}|x|^{N+\sigma}-k_{ij}\right|\le\omega(|x|) \quad\text{for } |x|\le1.
    $$
\end{itemize}
Here let us point out that a typical model fitting the condition $(A_3)$ is
$$
H(x,\xi)=h(x)|\xi|^q,
\quad \text{or more generally}\quad
H(x,\xi)=\sum^m_{i=1}h_i(x)|\xi|^{q_i},
$$
with bounded functions $h,h_i$, $i=1,2,\ldots,m$, and $q=\max\{q_1,q_2,\ldots,q_m\}$. The last assumption aforementioned ensures the nonlocal elliptic operator $\mathcal{I}_\sigma$ approaches a fully nonlinear (local) operator $F$ with uniform $(\lambda,\Lambda)$-ellipticity when $\sigma$ tends to 2.

In the case $\Phi(x,Du)=\gamma(|Du|)$, the conditions ($A_1$), ($A_2$) and ($A_4$) on $u,f,\mathcal{I}_\sigma$ are still satisfied separately.
Next we present the conditions required on the degeneracy law $\gamma$ and on the Hamiltonian term $H$:
\begin{itemize}
\item[($A_5$)] $\gamma:[0,+\infty)\rightarrow[0,+\infty)$ is a modulus of continuity carrying $\gamma(1)\ge1$ and its inverse $\gamma^{-1}$ fulfills the Dini condition;
\item[($A_6$)] $H:B_1\times \mathbb{R}^N\rightarrow \mathbb{R}$ is continuous and there is a constant $\mathcal{M}$ such that
    \begin{equation}
    \label{5-1}
    |H(x,\xi)|\le \mathcal{M}(1+\gamma(|\xi|))
    \end{equation}
    for every $x\in B_1$ and $\xi\in \mathbb{R}^N$.
\end{itemize}

Here the assumption $\gamma(1)\ge1$ is just a normalization. By modulus of continuity we mean an increasing function $\omega=\omega(t)$ defined on an interval $(0,T]$, or on $\mathbb{R}^+_0=(0,+\infty)$, into $\mathbb{R}^+_0$ satisfying $\lim_{t\rightarrow0}\omega(t)=0$. We say that a modulus of continuity $\omega$ meets Dini condition if
\begin{equation*}
\int^\tau_0\frac{\omega(t)}{t}\,dt<\infty
\end{equation*}
for some $\tau>0$.

Throughout this work, we denote as $C$ a positive constant that may vary from line to line. Relevant dependencies on parameters are indicated in parentheses, i.e., $C(p,q,N)$ means that $C$ depends on $p,q,N$. Besides, we say that a constant is universal if it depends at most upon the structure constants above. 

Next, we give the notion of viscosity solutions. If there is no danger of a confusion, we shall drop the subscript $\sigma$ of the nonlocal operator $\mathcal{I}_\sigma$ in the sequel.s
For generality, we define the solutions to the first order perturbation of \eqref{main}, i.e.,
\begin{equation}
\label{main2}
-\Phi(x,Du+\xi)\mathcal{I}_\sigma(u,x)+H(x,Du+\xi)=f(x) \quad\text{in }  B_1, 
\end{equation}
which will be utilized in proving the regularity properties of solutions to \eqref{main}.
For $u:\mathbb{R}^N\rightarrow\mathbb{R}$ and $K\in\mathcal{K}$, let
$$
I_{K}[\Omega](u,x)=C_{N,\sigma}\mathrm{P.V.}\int_\Omega(u(x+y)-u(x))K(y)\,dy
$$
where $\Omega\subseteq\mathbb{R}^N$ is a measurable set and the principal value coincides with the usual integral whenever $\mathrm{dist}(0,\Omega)>0$.
We say that a solution $u$ is normalized if $\|u\|_{L^\infty(B_1)}\le1$.

\begin{definition}[viscosity solutions]
\label{def1}
A function $u\in C(\overline{B}_1)\cap L^1_\sigma(\mathbb{R}^N)$ is a viscosity subsolution (respectively, supersolution) to \eqref{main2},
if for any $x_0\in B_1$ and any test function $\varphi(x)\in C^2(\mathbb{R}^N)$ such that $u-\varphi$ attains a local maximum (respectively, minimum) at $x_0$, we have
$$
-\Phi(x_0,D\varphi(x_0)+\xi)\mathcal{I}^\delta(u,\varphi,x_0)+H(x_0,D\varphi(x_0)+\xi)\le (\ge) f(x_0)
$$
where
$$
\mathcal{I}^\delta(u,\varphi,x_0)=\inf_j\sup_i\left(I_{K_{ij}}[B_\delta](\varphi,x_0)+I_{K_{ij}}[B^c_\delta](u,x_0)\right).
$$
Here $\{K_{ij}\}_{ij}\subset \mathcal{K}$ is a family of kernels, $B_\delta$ is a neighborhood of $x_0$ and $B^c_\delta$ denotes the complement of $B_\delta$ in $\mathbb{R}^N$.
A function is a viscosity solution to \eqref{main2} if it is both a viscosity subsolution and a supersolution.
\end{definition}

\begin{remark}
Alternatively, in the definition above, the requirements of viscosity subsolutions can be substituted with the following: for any $\varphi\in C^2(B_\delta(x_0))$ with any $B_\delta(x_0)\subset B_1$ such that $u-\varphi$ reaches a local maximum at $x_0$, we have
$$
-\Phi(x_0,D\varphi(x_0)+\xi)\mathcal{I}(\varphi_\delta,x_0)+H(x_0,D\varphi(x_0)+\xi)\le f(x_0)
$$
with the test function
\begin{equation*}
\varphi_\delta(x)=\begin{cases}\varphi(x), \quad x\in B_\delta(x_0),\\
u(x),  \quad x\in B^c_\delta(x_0),
\end{cases}
\end{equation*}
where
$$
\mathcal{I}(\varphi_\delta,x_0)=\inf_j\sup_i\left(I_{K_{ij}}[B_\delta](\varphi,x_0)+I_{K_{ij}}[B^c_\delta](u,x_0)\right).
$$
\end{remark}

We conclude this section by providing an important lemma that shall be employed several times later.

\begin{lemma}
\label{lem2-1}
Consider the cone $\mathcal{D}=\left\{z\in B_{\delta_0|\overline{a}|}:|\langle\overline{a},z\rangle|\ge(1-\eta_0)|\overline{a}||z|\right\}$ with $\overline{a}=\overline{x}-\overline{y}$, $0<|\overline{a}|\le\frac{1}{8}$ and $\delta_0,\eta_0\in\left(0,\frac{1}{2}\right)$. Let
$$
\delta^2u(x,z)=u(x)-u(x+z)+Du(x)\cdot z, \quad  x,z\in\mathbb{R}^N
$$
for a differentiable function $u:\mathbb{R}^N\rightarrow\mathbb{R}$. Set $\phi(x,y)=\varphi(|x-y|)$.
\begin{itemize}
\item[(1)] If
\begin{equation*}
\varphi(t)=\begin{cases}t-\frac{1}{4}t^{1+\beta}, \quad t\in [0,t_0],\\
\varphi(t_0),  \quad t\in (t_0,\infty),
\end{cases}
\end{equation*}
for some $t_0\in(0,1)$ and $\beta\in(0,1)$, then there is an absolute constant $\varepsilon\in\left(0,\frac{1}{2}\right)$ such that, for $\eta_0,\delta_0\le\min\{\varepsilon, \delta_1|\overline{a}|^\beta\}$ with $\delta_1=\frac{\beta(1+\beta)}{384}$, we obtain
\begin{equation*}
\frac{\beta(1+\beta)}{64}|\overline{a}|^{\beta-1}|z|^2\le \delta^2\phi(\cdot,\overline{y})(\overline{x},z)\le\frac{\beta(1+\beta)}{4}|\overline{a}|^{\beta-1}|z|^2
\end{equation*}
for every $z\in \mathcal{D}$.

\item[(2)] If $\varphi(t)=t^\gamma$ with $\gamma\in(0,1)$, then there exists $\varepsilon\in\left(0,\frac{1}{2}\right)$, depending only on $\gamma$, such that for any $\eta_0,\delta_0\le(0, \varepsilon]$
we infer that
\begin{equation*}
\frac{\gamma(1-\gamma)2^\gamma}{32}|\overline{a}|^{\gamma-2}|z|^2\le \delta^2\phi(\cdot,\overline{y})(\overline{x},z)\le\gamma(1-\gamma)|\overline{a}|^{\gamma-2}|z|^2
\end{equation*}
for every $z\in \mathcal{D}$.
\end{itemize}
\end{lemma}

\begin{proof}
(1) By means of Taylor's expansion, for $z\in \mathcal{D}$ we get
$$
\phi(\overline{x}+z,\overline{y})=\phi(\overline{x},\overline{y})+D_x\phi(\overline{x},\overline{y})\cdot z+\frac{1}{2}\langle z,D^2_x\phi(\overline{x}+tz,\overline{y})z\rangle,
$$
with $|t|\le 1$, and further
$$
\delta^2\phi(\cdot,\overline{y})(\overline{x},z)=-\frac{1}{2}\langle z,D^2_x\phi(\overline{x}+tz,\overline{y})z\rangle.
$$
A direct computation shows that
\begin{equation}
\label{2-2}
\langle z,D^2_x\phi(\overline{x}+tz,\overline{y})z\rangle
=\frac{\varphi'(|\overline{a}+tz|)}{|\overline{a}+tz|}\bigl(|z|^2-\langle\widehat{\overline{a}+tz},z\rangle^2\bigr)
+\varphi''(|\overline{a}+tz|)\langle\widehat{\overline{a}+tz},z\rangle^2
\end{equation}
with $\hat{x}=\frac{x}{|x|}$. For $z\in\mathcal{D}$ and $|t|\le 1$, we obtain
\begin{equation}
\label{2-3}
|\langle\overline{a}+tz,z\rangle|\ge(1-\eta_0)|\overline{a}||z|-|t||z|^2\ge(1-\eta_0-\delta_0)|\overline{a}||z|
\end{equation}
and
\begin{equation}
\label{2-4}
(1-\delta_0)|\overline{a}|\le |\overline{a}+tz|\le(1+\delta_0)|\overline{a}|.
\end{equation}
Let $\varepsilon\in\left(0,\frac{1}{2}\right)$ be small enough such that that, for $\eta_0,\delta_0\in(0,\varepsilon]$, we have
$$
\eta=\frac{1-\eta_0-\delta_0}{1+\delta_0}\ge\frac{1-2\varepsilon}{1+\varepsilon}\ge\frac{1}{\sqrt{2}}
$$
and
\begin{align*}
\frac{1-\eta^2}{1-\delta_0}
&=\frac{1}{1-\delta_0}\frac{2\delta_0+\eta_0}{1+\delta_0}\frac{2-\eta_0}{1+\delta_0}
=\frac{2\delta_0+\eta_0}{1-\delta_0}\frac{2-\eta_0}{(1+\delta_0)^2}\\
&\le2\frac{2\delta_0+\eta_0}{1-\delta_0}\le2\frac{2\delta_0+\eta_0}{1-\varepsilon}\le4(2\delta_0+\eta_0).
\end{align*}
Notice that
$$
\varphi'(|\overline{a}|)=1-\frac{1+\beta}{4}|\overline{a}|^\beta>0 \quad \text{and}\quad \varphi''(|\overline{a}|)=-\frac{\beta(1+\beta)}{4}|\overline{a}|^{\beta-1}<0.
$$

Then via \eqref{2-2}--\eqref{2-4} we arrive at
\begin{align*}
\varphi''(|\overline{a}+tz|)|z|^2&\le \langle z,D^2_x\phi(\overline{x}+tz,\overline{y})z\rangle\\
&\le\frac{|z|^2}{(1-\delta_0)|\overline{a}|}\left(1-\frac{(1-\eta_0-\delta_0)^2}{(1+\delta_0)^2}\right)-\frac{\beta(1+\beta)}{4}|\overline{a}+tz|^{\beta-1}
\langle\widehat{\overline{a}+tz},z\rangle^2\\
&=\frac{|z|^2}{(1-\delta_0)|\overline{a}|}(1-\eta^2)-\frac{\beta(1+\beta)}{4}|\overline{a}+tz|^{\beta-1}
\langle\widehat{\overline{a}+tz},z\rangle^2.
\end{align*}
Let us consider
$$
\langle\widehat{\overline{a}+tz},z\rangle^2\ge\frac{(1-\eta_0-\delta_0)^2|\overline{a}|^2|z|^2}{(1+\delta_0)^2|\overline{a}|^2}\ge\frac{1}{2}|z|^2
$$
and
\begin{align*}
-2|\overline{a}|^{\beta-1}
&\le-(1-\delta_0)^{\beta-1}|\overline{a}|^{\beta-1}\le-|\overline{a}+tz|^{\beta-1}\\
&\le-(1+\delta_0)^{\beta-1}|\overline{a}|^{\beta-1}\le-2^{\beta-1}|\overline{a}|^{\beta-1}\le-2^{-1}|\overline{a}|^{\beta-1}.
\end{align*}
It follows from the previous three displays that
\begin{align*}
-\frac{\beta(1+\beta)}{2}|\overline{a}|^{\beta-1}|z|^2&\le\langle z,D^2_x\phi(\overline{x}+tz,\overline{y})z\rangle\\
&\le4(\eta_0+2\delta_0)\frac{|z|^2}{|\overline{a}|}-\frac{\beta(1+\beta)}{16}\frac{|z|^2}{|\overline{a}|^{1-\beta}}\\
&=\frac{|z|^2}{|\overline{a}|}\left(4(\eta_0+2\delta_0)-\frac{\beta(1+\beta)}{16}|\overline{a}|^{\beta}\right).
\end{align*}
Let $\eta_0,\delta_0\le\delta_1|\overline{a}|^{\beta}$. Then
$$
4(\eta_0+2\delta_0)-\frac{\beta(1+\beta)}{16}|\overline{a}|^{\beta}\le\left(12\delta_1-\frac{\beta(1+\beta)}{16}\right)|\overline{a}|^{\beta}.
$$
Finally, letting $12\delta_1=\frac{\beta(1+\beta)}{32}$, we arrive at
$$
\frac{\beta(1+\beta)}{64}|\overline{a}|^{\beta-1}|z|^2\le\delta^2\phi(\cdot,\overline{y})(\overline{x},z)\le\frac{\beta(1+\beta)}{4}|\overline{a}|^{\beta-1}|z|^2.
$$

(2) Analogous to the proof of (1), we may conclude that
$$
-2\gamma(1-\gamma)|\overline{a}|^{\gamma-2}|z|^2\le\langle z,D^2_x\phi(\overline{x}+tz,\overline{y})z\rangle\le-\frac{\gamma(1-\gamma)2^\gamma}{8}|\overline{a}|^{\gamma-2}|z|^2
$$
for $|t|\le1, z\in\mathcal{D}$, and then get this assertion in Lemma \ref{lem2-1} (2). For the details, one also can refer to \cite[Lemma 2.2 (i)]{BT25}.
\end{proof}

\section{Lipschitz continuity of solutions}
\label{sec3}

In Section \ref{sec3} and Section \ref{sec4}, we consider the case that $\Phi(x,Du)=|Du|^p$. This section is mainly devoted to providing a compactness result for proving the approximation result in Lemma \ref{lem3-1}. To be precise, we are ready to establish the Lipschitz continuity for viscosity solutions to \eqref{main2} under the growth condition \eqref{2-1} on the Hamiltonian term with $0\le q\le p+\sigma$ ($\sigma\in(1,2)$), making use of the nonlocal version of well-known Ishii--Lions method developed by \cite{BCCI12, BCI11}. The case $p<q\le p+\sigma$ is technically more challenging than the case $0\le q\le p$, 
because we have to analyze the behaviour of \eqref{main2} when the norm of vector $\xi$ is large.

First we demonstrate the Lipschitz continuity of viscosity solutions to \eqref{main2} for the range $0\le q\le p$ in \eqref{2-1}.

\begin{lemma}
\label{lem3-1}
Let $\sigma\in(1,2)$ and assume that the conditions $(A_1)$--$(A_3)$ hold with $0\le q\le p$. Let $u\in C(\overline{B}_1)$ be a viscosity solution to \eqref{main2} with $\xi\in\mathbb{R}^N$. Then $u$ is locally Lipschitz continuous in $B_1$, that is, there exists a universal constant $C_{\rm lip}\ge1$, independent of $\xi$, such that
$$
|u(x)-u(y)|\le C_{\rm lip}|x-y|
$$
for every $x,y\in B_{\frac12}$. Furthermore, the $C_{\rm lip}$ is uniformly bounded as $\sigma\rightarrow2$.
\end{lemma}

\begin{proof}
We discuss the cases when $|\xi|$ is large or small separately.

\textbf{Case 1.} $|\xi|$ is large. Suppose that $|\xi|\ge A$ for a number $A>0$ to be chosen later. Consider two smooth functions
\begin{equation}
\label{3-1-1}
\psi(x)=m\left[\left(|x|^2-\frac{1}{4}\right)_+\right]^2 \quad\text{and}\quad \omega(t)=\begin{cases}t-\frac{1}{4}t^{1+\beta},\qquad0\le t\le1,\\
\omega(1),\quad t>1,
\end{cases}
\end{equation}
where $m>0$ and $\beta\in(0,1)$ are parameters to be fixed later. We proceed by utilizing the doubling variables method.
We introduce the auxiliary functions
$$
\phi(x,y)=L\omega(|x-y|)+\psi(y)
$$
and
$$
\Psi(x,y)=u(x)-u(y)-\phi(x,y)
$$
with $L\ge1$ to be determined later. By continuity of the function $\Psi$, there exists a point   $(\overline{x},\overline{y})\in \overline{B}_1\times\overline{B}_1$ such that $\Psi(\overline{x},\overline{y})=\sup_{B_1\times B_1}\Psi(x,y)$. We show that $\Psi(\overline{x},\overline{y})\le0$, which implies the local Lipschitz continuity.
For a contradiction, assume that  $\Psi(\overline{x},\overline{y})>0$. Then
\begin{equation}
\label{3-1-2}
\begin{split}
0&<u(\overline{x})-u(\overline{y})-L\omega(|\overline{x}-\overline{y}|)-\psi(\overline{y}) \\
&\le2\|u\|_{L^\infty(B_1)}-L\omega(|\overline{x}-\overline{y}|)-m\left[\left(|\overline{y}|^2-\frac{1}{4}\right)_+\right]^2.
\end{split}
\end{equation}
Let $m$ be large enough that
$$
m\left[\left(|y|^2-\frac{1}{4}\right)_+\right]^2>2\|u\|_{L^\infty(B_1)} \quad\text{for every } |y|\ge\frac{3}{4}.
$$
It follows that
$$
m\left[\left(\frac{3}{4}\right)^2-\frac{1}{4}\right]^2=m\left(\frac{5}{16}\right)^2>2\|u\|_{L^\infty(B_1)}
$$
and consequently
$$
m>2\left(\frac{16}{5}\right)^2\|u\|_{L^\infty(B_1)}.
 $$
Thus $\overline{y}\in B_{\frac34}$. Let $L\ge1$ be large enough that
$$
2\|u\|_{L^\infty(B_1)}<L\left(\frac{1}{8}-\frac{1}{4}\left(\frac{1}{8}\right)^{1+\beta}\right).
$$
Then
$$
L>16\left(1-\frac{1}{4}\left(\frac{1}{8}\right)^\beta\right)^{-1}\|u\|_{L^\infty(B_1)},
$$
which implies that $|\overline{x}-\overline{y}|<\frac{1}{8}$
and further $|\overline{x}|\le |\overline{x}-\overline{y}|+|\overline{y}|<\frac{7}{8}$.
It follows that $(\overline{x},\overline{y})\in B_\frac{7}{8}\times B_\frac{7}{8}\subset B_1\times B_1$ and $\overline{x}\neq\overline{y}$.
Moreover, by \eqref{3-1-2} we obtain
$$
L\omega(|\overline{x}-\overline{y}|)\le2\|u\|_{L^\infty(B_1)}
$$
which implies that
$$
 |\overline{x}-\overline{y}|\le\frac{4\|u\|_{L^\infty(B_1)}}{L}.
$$

Because $(\overline{x},\overline{y})$ is the maximum point, we have
\begin{equation}
\label{3-1-3}
\begin{cases}
u(x)\le \phi(x,\overline{y})-\phi(\overline{x},\overline{y})+u(\overline{x})=h(x),\quad  x\in B_1,\\
u(y)\ge \phi(\overline{x},\overline{y})-\phi(\overline{x},y)+u(\overline{y})=\eta(y),\quad  y\in B_1.
\end{cases}
\end{equation}
For any $\delta\in\left(0,\frac{1}{8}\right)$, consider the test functions
\begin{equation*}
v(z)=\begin{cases}h(z),\quad z\in B_\delta(\overline{x}),\\
u(z)\quad  \text{ otherwise},
\end{cases}
\end{equation*}
and
\begin{equation*}
w(z)=\begin{cases}\eta(z),  \quad z\in B_\delta(\overline{y}),\\
u(z) \quad\text{ otherwise}.
\end{cases}
\end{equation*}
It follows from the definition of viscosity solution that
\begin{equation}
\label{3-1-3-1}
\begin{cases}
-|Dh(\overline{x})+\xi|^p\mathcal{I}_\sigma(v,\overline{x})+H(\overline{x},Dh(\overline{x})+\xi)\le f(\overline{x}),\\
-|D\eta(\overline{y})+\xi|^p\mathcal{I}_\sigma(w,\overline{y})+H(\overline{y},D\eta(\overline{y})+\xi)\ge f(\overline{y}).
\end{cases}
\end{equation}
We note that
$$
Dh(\overline{x})=L\omega'(|\overline{x}-\overline{y}|)\frac{\overline{x}-\overline{y}}{|\overline{x}-\overline{y}|},
 \quad\text{with}\quad \omega'(t)=1-\frac{1+\beta}{4}t^\beta,
$$
and
$$
D\eta(\overline{y})=-L\omega'(|\overline{x}-\overline{y}|)\frac{\overline{x}-\overline{y}}{|\overline{x}-\overline{y}|}+D\psi(\overline{y}),
\quad\text{with}\quad D\psi(\overline{y})=4m\Big(|\overline{y}|^2-\frac{1}{4}\Big)_+\overline{y}.
$$
Moreover, we have
\begin{equation}
\label{3-1-3-2}
\frac{L}{2}\le|Dh(\overline{x})|\le L
\end{equation}
and
\begin{equation}
\label{3-1-3-3}
\frac{L}{4}\le\frac{L}{2}-4m\Big(|\overline{y}|^2-\frac{1}{4}\Big)_+|\overline{y}|\le|D\eta(\overline{y})|\le L+4m\Big(|\overline{y}|^2-\frac{1}{4}\Big)_+|\overline{y}|\le\frac{5L}{4}
\end{equation}
for large enough $L\ge12m$.

Set $A=2L$. By the assumption $|\xi|\ge A$, we have
\begin{equation*}
\begin{cases}
|Dh(\overline{x})+\xi|\ge |\xi|-|Dh(\overline{x})|\ge L,   \\
|D\eta(\overline{y})+\xi|\ge |\xi|-|D\eta(\overline{y})|\ge \frac{3}{4}L.
\end{cases}
\end{equation*}
Applying the viscosity inequalities above \eqref{3-1-3-1} and the growth condition \eqref{2-1} on $H(\cdot)$, we arrive that
\begin{align}
\label{3-1-4}
&\mathcal{I}_\sigma(v,\overline{x})-\mathcal{I}_\sigma(w,\overline{y}) \nonumber\\
&\qquad\ge\frac{f(\overline{y})}{|D\eta(\overline{y})+\xi|^p}-
\frac{f(\overline{x})}{|D\eta(\overline{x})+\xi|^p}+\frac{H(\overline{x},D\eta(\overline{x})+\xi)}{|D\eta(\overline{x})+\xi|^p}-
\frac{H(\overline{y},D\eta(\overline{y})+\xi)}{|D\eta(\overline{y})+\xi|^p} \nonumber\\
&\qquad\ge-C(p)\frac{\|f\|_{L^\infty(B_1)}}{L^p}-\left(\frac{\mathcal{M}+\mathcal{H}|Dh(\overline{x})+\xi|^q}{|Dh(\overline{x})+\xi|^p}+
\frac{\mathcal{M}+\mathcal{H}|D\eta(\overline{y})+\xi|^q}{|D\eta(\overline{y})+\xi|^p}\right) \nonumber\\
&\qquad\ge-C(p)\frac{\mathcal{M}+\|f\|_{L^\infty(B_1)}}{L^p}-\mathcal{H}\left(|Dh(\overline{x})+\xi|^{q-p}+|D\eta(\overline{y})+\xi|^{q-p}\right) \nonumber\\
&\qquad\ge-C(p,q)(\mathcal{M}+\mathcal{H}+\|f\|_{L^\infty(B_1)}),
\end{align}
where in the last line we have used the facts that $L\ge1$ and $q-p\le0$.

Next, let us evaluate the left-hand side of \eqref{3-1-4}. As in (2.6) in \cite{PT21} or (14) in \cite{APS24}, we obtain \eqref{3-1-5} below. For every $|z|\le\frac{1}{8}$, we find $\overline{x}+z,\overline{y}+z\in B_1$. Then there exists a kernel $K$ in the family $\mathcal{K}_0$ fulfilling
\begin{equation}
\label{3-1-5}
-C(\|f\|_{L^\infty(B_1)}+\|u\|_{L^\infty(B_1)}+\|u\|_{L^1_\sigma}+1)\le J_1+J_2
\end{equation}
with
$$
J_1=I_K[B_\delta](\phi(\cdot,\overline{y}),\overline{x})-I_K[B_\delta](-\phi(\overline{x},\cdot),\overline{y})
$$
and
$$
J_2=I_K[B_{\frac18}\setminus B_\delta](u,\overline{x}))-I_K[B_{\frac18}\setminus B_\delta](u,\overline{y}).
$$
Here $C\ge1$ is a universal constant (depending on $\mathcal{M},\mathcal{H}$).
Let $\overline{a}=\overline{x}-\overline{y}$ and
$$
\mathcal{D}=\left\{z\in B_{\delta_0|\overline{a}|}:|\langle\overline{a},z\rangle|\ge(1-\eta_0)|\overline{a}||z|\right\}
$$
with $\delta_0|\overline{a}|<\frac{1}{8}$ and $\eta_0\in(0,1)$ to be chosen later. For $z\in B_\frac{1}{8}$, from \eqref{3-1-3} we obtain
\begin{equation*}
\begin{cases}
u(\overline{x}+z)-u(\overline{x})\le \phi(\overline{x}+z,\overline{y})-\phi(\overline{x},\overline{y}),\\
u(\overline{y}+z)-u(\overline{y})\ge -\phi(\overline{x},\overline{y}+z)+\phi(\overline{x},\overline{y}),
\end{cases}
\end{equation*}
and further
\begin{equation*}
\begin{cases}
I_K[\mathcal{D}\setminus B_\delta](u,\overline{x}))\le LI_K[\mathcal{D}\setminus B_\delta](\omega,\overline{a}),\\
I_K[\mathcal{D}\setminus B_\delta](u,\overline{y}))\ge -LI_K[\mathcal{D}\setminus B_\delta](\omega,\overline{a})-Cm,
\end{cases}
\end{equation*}
where we employed the smoothness of $\psi$, and the constant $C\ge1$ is uniformly bounded as $\delta\rightarrow0$. In addition, by means of
$$
u(\overline{x}+z)-u(\overline{x})-(u(\overline{y}+z)-u(\overline{y}))\le \phi(\overline{x}+z,\overline{y}+z)-\phi(\overline{x},\overline{y}),
$$
we deduce for $\mathcal{O}=(B_{\frac18}\setminus B_\delta)\setminus(\mathcal{D}\setminus B_\delta)$ that
$$
I_K[\mathcal{O}](u,\overline{x})-I_K[\mathcal{O}](u,\overline{y})\le I_K[\mathcal{O}](\psi,\overline{y})\le Cm,
$$
where the latter inequality follows from the smoothness of $\psi$. At this point, by combining the inequalities above with \eqref{3-1-5}, we obtain
\begin{equation}
\label{3-1-6}
-C(\|f\|_{L^\infty(B_1)}+\|u\|_{L^\infty(B_1)}+\|u\|_{L^1_\sigma}+1)\le J_1+2LI_K[\mathcal{D}\setminus B_\delta](\omega,\overline{a}).
\end{equation}
Here we note that $J_1\rightarrow0$ as $\delta\rightarrow0$, and $I_K[\mathcal{D}\setminus B_\delta](\omega,\overline{a})\rightarrow I_K[\mathcal{D}](\omega,\overline{a})$ as $\delta\rightarrow0$
by the dominated convergence theorem.

Then we concentrate on the term $I_K[\mathcal{D}](\omega,\overline{a})$. Let $l(z)=D_x\omega(|\overline{x}-\overline{y}|)\cdot z$. Then
\begin{align*}
I_K[\mathcal{D}](\omega,\overline{a})&=\int_{\mathcal{D}}(\omega(|\overline{a}+z|)-\omega(|\overline{a}|))K(z)\,dz-\int_{\mathcal{D}}
(l(\overline{x}+z)-l(\overline{x}))K(z)\,dz\\
&=\int_{\mathcal{D}}(\omega(|\overline{a}+z|)-\omega(|\overline{a}|)-D_x\omega(|\overline{x}-\overline{y}|)\cdot z)K(z)\,dz\\
&=-\int_{\mathcal{D}}\delta^2\omega(\cdot,\overline{y})(\overline{x},z)K(z)\,dz,
\end{align*}
with $\omega(x,y)=\omega(|x-y|)$ and $\delta^2\omega$ defined as in Lemma \ref{lem2-1}.
On the first line of the display above, we notice that the second integral equals to 0 by the symmetry of the domain $\mathcal{D}$. Next, we take $\eta_0=\delta_0=\delta_1|\overline{a}|^\beta$ with $\delta_1$ a very small positive number depending only upon $\beta$. In view of Lemma \ref{lem2-1} (1) and
$$
C_{N,\sigma}\frac{\lambda}{|y|^{N+\sigma}}\le K(y)\le C_{N,\sigma}\frac{\Lambda}{|y|^{N+\sigma}},
$$
we arrive at
\begin{align*}
I_K[\mathcal{D}](\omega,\overline{a})&\le -\frac{\beta(\beta+1)C_{N,\sigma}\lambda}{64}|\overline{a}|^{\beta-1}\int_{\mathcal{D}}|z|^{2-N-\sigma}\,dz\\
&\le-\frac{c\beta(\beta+1)C_{N,\sigma}\lambda}{64(2-\sigma)}|\overline{a}|^{\beta-1}\eta_0^\frac{N-1}{2}(\delta_0|\overline{a}|)^{2-\sigma}\\
&=-C|\overline{a}|^{\beta-1+\frac{N-1}{2}\beta+(\beta+1)(2-\sigma)},
\end{align*}
where in the second line we have exploited the estimate on $\int_{\mathcal{D}}|z|^{2-N-\sigma}\,dz$ in Example 1 in \cite{BCCI12}, and $C\ge1$ is a universal constant that is uniformly bounded as $\sigma\rightarrow2$. Note that we can select $\beta>0$ small enough to obtain
\begin{align*}
\vartheta&=\beta-1+\frac{N-1}{2}\beta+(\beta+1)(2-\sigma)\\
&=1-\sigma+\beta\left(3+\frac{N-1}{2}-\sigma\right)\le \frac{1-\sigma}{2}<0
\end{align*}
for $\sigma\in(1,2)$.

As a result, by taking $\eta_0=\delta_0=\delta_1|\overline{a}|^\beta$ and sending $\delta\rightarrow0$ in \eqref{3-1-6}, it follows that
$$
-C(\|f\|_{L^\infty(B_1)}+\|u\|_{L^\infty(B_1)}+\|u\|_{L^1_\sigma}+1)\le -CL|\overline{a}|^\vartheta\le -CL,
$$
which indicates a contradiction for $L\ge1$ large enough.

\textbf{Case 2.} $|\xi|$ is small. Let $|\xi|<A$. Since the argument is similar to Case 1, we only sketch it. Let
$$
\phi(x,y)=\tilde{L}\omega(|x-y|)+\psi(y).
$$
We continue as in Case 1, with $L$ substituted by $\tilde{L}$. Then \eqref{3-1-3-2} and \eqref{3-1-3-3} turn into
$$
\frac{\tilde{L}}{4}\le |Dh(\overline{x})|,|D\eta(\overline{y})|\le \frac{5}{4}\tilde{L}.
$$
Since $|\xi|<A$, we have
\begin{equation*}
\begin{cases}
|Dh(\overline{x})+\xi|\ge|Dh(\overline{x})|-|\xi|\ge\frac{\tilde{L}}{4}-A\ge A,\\
|D\eta(\overline{y})+\xi|\ge\frac{5\tilde{L}}{4}-A\ge 9A,
\end{cases}
\end{equation*}
by taking $\tilde{L}\ge8A$. Moreover, we also have
$$
\mathcal{I}_\sigma(v,\overline{x})-\mathcal{I}_\sigma(w,\overline{y})\ge-C(\mathcal{M}+\mathcal{H}+\|f\|_{L^\infty(B_1)}).
$$
By the same proof in Case 1, we obtain the Lipschitz continuity of $u$ by choosing $\tilde{L}$ sufficiently large.
\end{proof}

In the scenario $p<q\le p+1$, we need pay special attention to restraining the growth of the term $|Du+\xi|^{q-p}$, but the idea is similar to that of Lemma \ref{lem3-1} as well.

\begin{lemma}
\label{lem3-2}
Let $\sigma\in(1,2)$ and assume that the conditions $(A_1)$--$(A_3)$ hold with $p< q\le p+1$. Let $u\in C(\overline{B}_1)$ be a viscosity solution to \eqref{main2} with $\xi\in\mathbb{R}^N$.
Then there exists a universal constant $b_0>0$ such that if
\begin{align}
\label{3-2-1}
\mathcal{H}(1+|\xi|^{q-p})\le b
\end{align}
for some $b\le b_0$, then $u$ is locally H\"{o}lder continuous in $B_1$.
In other words, there exist universal constants $\alpha\in(0,1)$ and $C\ge1$ (independent of $\xi$) satisfying
$$
|u(x)-u(y)|\le C|x-y|^\alpha
$$
for every $x,y\in B_{\frac12}$. Furthermore, the constant $C$ is uniformly bounded as $\sigma\rightarrow2$.
\end{lemma}

\begin{proof}
Let
$$
\theta=\left(\frac{b_0}{\mathcal{H}}\right)^\frac{1}{q-p}.
$$
In view of \eqref{3-2-1}, we have
\begin{equation}
\label{3-2-2}
\theta\ge1 \quad\text{and}\quad |\xi|\le \theta.
\end{equation}
The proof is similar to that of Lemma \ref{lem3-1}. We substitute $\omega(t)$ in \eqref{3-1-1} in Lemma \ref{lem3-1} with
$\varphi(t)=t^\alpha$,  $t\in[0,\infty)$ for $\alpha\in(0,1)$ to be fixed. By doubling variables we introduce two functions
\begin{equation*}
\begin{cases}
\widetilde{\phi}(x,y)=L\varphi(|x-y|)+\psi(y),\\
\Psi(x,y)=u(x)-u(y)-\widetilde{\phi}(x,y).
\end{cases}
\end{equation*}
Let $\Psi(\overline{x},\overline{y})=\sup_{B_1\times B_1}\Psi(x,y)$ with $(\overline{x},\overline{y})\in \overline{B}_1\times\overline{B}_1$. We show that$\Psi(\overline{x},\overline{y})\le0$
in order to conclude the local H\"{o}lder continuity.
For a contradiction, assume that $\Psi(\overline{x},\overline{y})>0$.
Then
$$
L|\overline{x}-\overline{y}|^\alpha+m\left[\left(|\overline{y}|^2-\frac{1}{4}\right)_+\right]^2\le2\|u\|_{L^\infty(B_1)}.
$$
We may choose
$$
m>2\left(\frac{16}{5}\right)^2\|u\|_{L^\infty(B_1)}
\quad\text{and}\quad
L>16\|u\|_{L^\infty(B_1)}
$$
such that $|\overline{y}|<\frac{3}{4}$, $|\overline{x}-\overline{y}|<\frac{1}{8}$.
It follows that $(\overline{x},\overline{y})\in B_\frac{7}{8}\times B_\frac{7}{8}\subset B_1\times B_1$ and $\overline{x}\neq\overline{y}$, as well as
\begin{equation}
\label{3-2-3}
|\overline{x}-\overline{y}|\le\left(\frac{2\|u\|_{L^\infty(B_1)}}{L}\right)^\frac{1}{\alpha}.
\end{equation}

As in the proof of Lemma \ref{lem3-1}, we obtain the viscosity inequalities
\begin{equation}
\label{3-2-4}
\begin{cases}
-|D\widetilde{h}(\overline{x})+\xi|^p\mathcal{I}_\sigma(\widetilde{v},\overline{x})+H(\overline{x},D\widetilde{h}(\overline{x})+\xi)\le f(\overline{x}),\\
-|D\widetilde{\eta}(\overline{y})+\xi|^p\mathcal{I}_\sigma(\widetilde{w},\overline{y})+H(\overline{y},D\widetilde{\eta}(\overline{y})+\xi)\ge f(\overline{y}),
\end{cases}
\end{equation}
where
\begin{equation*}
\begin{cases}
\widetilde{h}(x)=\widetilde{\phi}(x,\overline{y})-\widetilde{\phi}(\overline{x},\overline{y})+u(\overline{x}),\\
\widetilde{\eta}(y)=-\widetilde{\phi}(\overline{x},y)+\widetilde{\phi}(\overline{x},\overline{y})+u(\overline{y}),
\end{cases}
\end{equation*}
along with
\begin{equation*}
\widetilde{v}(z)=\begin{cases}\widetilde{h}(z),\quad z\in B_\delta(\overline{x}),\\
u(z)\quad\text{ otherwise},
\end{cases}
\end{equation*}
and
\begin{equation*}
\widetilde{w}(z)=\begin{cases}\widetilde{\eta}(z),\quad z\in B_\delta(\overline{y}),\\
u(z)\quad\text{otherwise},
\end{cases}
\end{equation*}
for $\delta\in(0,\frac18)$. A direct calculation shows that
\begin{equation*}
\begin{cases}
D\widetilde{h}(\overline{x})=\alpha L|\overline{x}-\overline{y}|^{\alpha-2}(\overline{x}-\overline{y}),\\
D\widetilde{\eta}(\overline{y})=-\alpha L|\overline{x}-\overline{y}|^{\alpha-2}(\overline{x}-\overline{y})+4m\left(|\overline{y}|^2-\frac{1}{4}\right)_+\overline{y},
\end{cases}
\end{equation*}
and further
\begin{equation*}
\begin{cases}
|D\widetilde{h}(\overline{x})|=\alpha L|\overline{x}-\overline{y}|^{\alpha-1}, \\
\frac{\alpha L|\overline{x}-\overline{y}|^{\alpha-1}}{2}\le|D\widetilde{\eta}(\overline{y})|\le2\alpha L|\overline{x}-\overline{y}|^{\alpha-1}
\end{cases}
\end{equation*}
for $L\ge\frac{15m}{16\alpha}$. Since
$$
4m\left(|\overline{y}|^2-\frac{1}{4}\right)_+|\overline{y}|\le 4m\frac{5}{16}\cdot\frac{3}{4},$$
we may take $L$ large enough that
$$
\frac{15m}{16}\le\alpha L<\alpha L|\overline{x}-\overline{y}|^{\alpha-1}.
$$
Moreover, by taking $L\ge4\alpha^{-1}\theta$, we obtain
$$
|D\widetilde{h}(\overline{x})+\xi|\le\alpha L|\overline{x}-\overline{y}|^{\alpha-1}+\theta\le 2\alpha L|\overline{x}-\overline{y}|^{\alpha-1}
$$
and
$$
|D\widetilde{h}(\overline{x})+\xi|\ge\alpha L|\overline{x}-\overline{y}|^{\alpha-1}-\theta\ge4\theta|\overline{x}-\overline{y}|^{\alpha-1}-\theta\ge3\theta|\overline{x}-\overline{y}|^{\alpha-1},
$$
where \eqref{3-2-2} and $|\overline{x}-\overline{y}|^{\alpha-1}>1$ were utilized. Analogously, we can treat $|D\widetilde{\eta}(\overline{y})+\xi|$. In summary, we have
\begin{equation}
\label{3-2-5}
1<\theta|\overline{x}-\overline{y}|^{\alpha-1}\le|D\widetilde{h}(\overline{x})+\xi|,|D\widetilde{\eta}(\overline{y})+\xi|\le3\alpha L|\overline{x}-\overline{y}|^{\alpha-1}.
\end{equation}
Merging \eqref{3-2-4}, \eqref{3-2-5} with the condition \eqref{2-1} on $H(\cdot)$ with $p<q\le p+1$, we arrive at
\begin{equation}
\label{3-2-6}
\begin{split}
&\mathcal{I}_\sigma(\widetilde{v},\overline{x})-\mathcal{I}_\sigma(\widetilde{w},\overline{y}) \\
&\qquad\ge-2\|f\|_{L^\infty(B_1)}-\left(\frac{\mathcal{M}+\mathcal{H}|D\widetilde{h}(\overline{x})+\xi|^q}{|D\widetilde{h}(\overline{x})+\xi|^p}+
\frac{\mathcal{M}+\mathcal{H}|D\widetilde{\eta}(\overline{y})+\xi|^q}{|D\widetilde{\eta}(\overline{y})+\xi|^p}\right) \\
&\qquad\ge-2(\mathcal{M}+\|f\|_{L^\infty(B_1)})-\mathcal{H}\left(|D\widetilde{h}(\overline{x})+\xi|^{q-p}+|D\widetilde{\eta}(\overline{y})+\xi|^{q-p}\right) \\
&\qquad\ge-2(\mathcal{M}+\|f\|_{L^\infty(B_1)})-\mathcal{H}(3\alpha L)^{q-p}|\overline{x}-\overline{y}|^{(\alpha-1)(q-p)}.
\end{split}
\end{equation}

The estimate on the term $\mathcal{I}_\sigma(\widetilde{v},\overline{x})-\mathcal{I}_\sigma(\widetilde{w},\overline{y})$ in the previous display is similar to Lemma \ref{lem3-1}
and we obtain
\begin{equation}
\label{3-2-7}
\begin{split}
&-C(\|f\|_{L^\infty(B_1)}+\|u\|_{L^\infty(B_1)}+\|u\|_{L^1_\sigma}+\mathcal{M})\\
&\qquad-\mathcal{H}(3\alpha L)^{q-p}|\overline{x}-\overline{y}|^{(\alpha-1)(q-p)}\le LI_K[\mathcal{D}](\varphi,\overline{a}),
\end{split}
\end{equation}
corresponding to \eqref{3-1-6} as $\delta\rightarrow0$.
Let $l(z)=D_x\varphi(|\overline{x}-\overline{y}|)\cdot z$. There holds
$$
I_K[\mathcal{D}](\varphi,\overline{a})=-\int_{\mathcal{D}}\delta^2\varphi(\cdot,\overline{y})(\overline{x},z)K(z)\,dz
$$
with $\varphi(x,y)=\varphi(|x-y|)$. By virtue of Lemma \ref{lem2-1} (2) and
$$
C_{N,\sigma}\frac{\lambda}{|y|^{N+\sigma}}\le K(y)\le C_{N,\sigma}\frac{\Lambda}{|y|^{N+\sigma}},
$$
we have
\begin{equation}
\label{3-2-8}
\begin{split}
I_K[\mathcal{D}](\varphi,\overline{a})&\le-\frac{\alpha(1-\alpha)2^\alpha C_{N,\sigma}\lambda}{32}|\overline{a}|^{\alpha-2}\int_{\mathcal{D}}
|z|^{2-N-\sigma}\,dz \\
&\le-C|\overline{a}|^{\alpha-2}\eta_0^{\frac{N-1}{2}}(\delta_0|\overline{a}|)^{2-\sigma}
=-\widetilde{C}|\overline{a}|^{\alpha-\sigma},
\end{split}
\end{equation}
where in the last line we have selected $\eta_0=\delta_0=\varepsilon$ with $\varepsilon\in(0,\frac{1}{2})$ a number coming from Lemma \ref{lem2-1} (2), and the constant $\widetilde{C}>0$ is uniformly bounded as $\sigma\rightarrow2$. It follows from \eqref{3-2-7}, \eqref{3-2-8} that
\[
\begin{split}
&C(\|f\|_{L^\infty(B_1)}+\|u\|_{L^\infty(B_1)}+\|u\|_{L^1_\sigma}+\mathcal{M})\\
&\qquad+\mathcal{H}(3\alpha L)^{q-p}|\overline{x}-\overline{y}|^{(\alpha-1)(q-p)}\ge\widetilde{C}L|\overline{x}-\overline{y}|^{\alpha-\sigma}
\end{split}
\]
and further by $0<\alpha<1<\sigma$
\begin{equation}
\label{3-2-9}
\begin{split}
\widetilde{C}L&\le C(\|f\|_{L^\infty(B_1)}+\|u\|_{L^\infty(B_1)}+\|u\|_{L^1_\sigma}+\mathcal{M})\\
&\qquad+\mathcal{H}(3\alpha L)^{q-p}|\overline{x}-\overline{y}|^{(\alpha-1)(q-p)+\sigma-\alpha}
\end{split}
\end{equation}
Observe that, since $q\le p+1$ and $\alpha<1$, we have
$$
(\alpha-1)(q-p)+\sigma-\alpha=(\alpha-1)(q-p-1)+\sigma-1>0.
$$
Therefore, recalling \eqref{3-2-1} and $q-p\le1$, we obtain
$$
\widetilde{C}L\le C(\|f\|_{L^\infty(B_1)}+\|u\|_{L^\infty(B_1)}+\|u\|_{L^1_\sigma}+\mathcal{M})+3b_0\alpha ^{q-p}L,
$$
i.e.,
$$
(\widetilde{C}-3b_0\alpha ^{q-p})L\le C(\|f\|_{L^\infty(B_1)}+\|u\|_{L^\infty(B_1)}+\|u\|_{L^1_\sigma}+\mathcal{M}).
$$
Let
$$
\alpha\le\left(\frac{1}{3}\right)^\frac{1}{q-p} \quad\text{and}\quad  0<b_0\le\frac{\widetilde{C}}{2}.
$$
Then
$$
\frac{\widetilde{C}}{2}L\le C(\|f\|_{L^\infty(B_1)}+\|u\|_{L^\infty(B_1)}+\|u\|_{L^1_\sigma}+\mathcal{M}).
$$
with $C,\widetilde{C}>0$ being two universal constants independent of $\mathcal{M},\mathcal{H}$. Taking sufficiently large $L$ leads to a contradiction from the preceding inequality.
\end{proof}

In the end of this section, we are ready to consider the case $p+1<q\le p+\sigma$. Unfortunately, based on the technology applied in the present paper, we can not verify the $C^{1,\alpha}$-regularity of viscosity solution, so we are unnecessary to examine the gradient perturbation equation \eqref{main2}. Then we simply study the Lipschitz continuity of solutions to Eq. \eqref{main}, which is of independent interest.

\smallskip

We first clarify the derivation of the upper constraint condition $q\le p+\sigma$ as follows. We follow the proof of Lemma \ref{lem3-2} verbatim, except taking $\xi=0$ this time, and arrive \eqref{3-2-9}. In order to derive a contradiction from \eqref{3-2-9} by choosing large enough $L$, we need require that the exponent $(\alpha-1)(q-p)+\sigma-\alpha$ is positive, which together with \eqref{3-2-3} may make the exponent of $L$ on the right-hand side of \eqref{3-2-9} not larger than 1. That is, we demand
$$
1\ge q-p-\frac{(\alpha-1)(q-p)+\sigma-\alpha}{\alpha},
$$
or equivalently $p+\sigma\ge q$.
It is worth mentioning that if $\sigma\rightarrow2$, the condition above coincides with the condition imposed by \cite{BDL19,BD16} for local fully nonlinear equation with a Hamiltonian term.
On the contrary, if the condition $q\le p+\sigma$ is violated, we cannot conclude the H\"{o}lder continuity by the classical Ishii--Lions method, since we cannot find an appropriate $L$ for the contradiction.

The forthcoming lemma states the H\"{o}lder regularity for the equation \eqref{main}, whose proof is very similar to that of Lemma \ref{lem3-2}.

\begin{lemma}
\label{lem3-3}
Let $\sigma\in(1,2)$ and assume that the conditions $(A_1)$--$(A_3)$ hold with $p+1< q\le p+\sigma$. Let $u\in C(\overline{B}_1)$ be a viscosity solution to \eqref{main}. Then $u$ is locally H\"{o}lder continuous in $B_1$. More precisely, there exist universal constants $\alpha\in(0,1)$ and $C\ge1$ such that
$$
|u(x)-u(y)|\le C|x-y|^\alpha
$$
for every $x,y\in B_{\frac12}$, where the constant $C$ is uniformly bounded as $\sigma\rightarrow2$.
\end{lemma}

\begin{proof}
We use the same notation as in the proof of Lemma \ref{lem3-2}.
As in the proof of Lemma \ref{lem3-2} with $\xi=0$, we arrive at \eqref{3-2-9} and rewrite it as
\begin{align*}
\widetilde{C}L
&\le C(\|f\|_{L^\infty(B_1)}+\|u\|_{L^\infty(B_1)}+\|u\|_{L^1_\sigma}+\mathcal{M})\\
&\qquad+\mathcal{H}(3\alpha L)^{q-p}|\overline{x}-\overline{y}|^{(\alpha-1)(q-p)+\sigma-\alpha}.
\end{align*}
For $p+1< q\le p+\sigma$, we have
\begin{align*}
q-p-\frac{(\alpha-1)(q-p)+\sigma-\alpha}{\alpha}&=q-p-\frac{(\alpha-1)(q-p-1)+\sigma-1}{\alpha}\\
&\le q-p-\frac{(\alpha-1)(\sigma-1)+\sigma-1}{\alpha}\\
&=q-p+1-\sigma\in(2-\sigma,1].
\end{align*}
Then by \eqref{3-2-3} and $p+1< q\le p+\sigma$, we obtain
\begin{align*}
\widetilde{C}L&\le C(\|f\|_{L^\infty(B_1)}+\|u\|_{L^\infty(B_1)}+\|u\|_{L^1_\sigma}+\mathcal{M})\\
&\qquad+\mathcal{H}(3\alpha L)^{q-p}\left(\frac{2\|u\|_{L^\infty(B_1)}}{L}\right)^\frac{(\alpha-1)(q-p)+\sigma-\alpha}{\alpha}\\
&\le C(\|f\|_{L^\infty(B_1)}+\|u\|_{L^\infty(B_1)}+\|u\|_{L^1_\sigma}+\mathcal{M})\\
&\qquad+(2\|u\|_{L^\infty(B_1)})^\gamma\mathcal{H}(3\alpha )^{q-p}L^{1+q-p-\sigma},
\end{align*}
with $\gamma=q-p-1+\frac{p+\sigma-q}{\alpha}$.
It follows that
\begin{equation}
\label{3-3-1}
\begin{split}
&(\widetilde{C}-(2\|u\|_{L^\infty(B_1)})^\gamma\mathcal{H}(3\alpha )^{q-p}L^{q-p-\sigma})L\\
&\qquad\le C(\|f\|_{L^\infty(B_1)}+\|u\|_{L^\infty(B_1)}+\|u\|_{L^1_\sigma}+\mathcal{M}).
\end{split}
\end{equation}
If $q\in(p+1,p+\sigma)$, we may select $L$ sufficiently large to obtain
$$
\widetilde{C}-(2\|u\|_{L^\infty(B_1)})^\gamma\mathcal{H}(3\alpha )^{q-p}L^{q-p-\sigma}\ge\frac{\widetilde{C}}{2}.
$$
By choosing $L$ large enough concludes a contradiction with \eqref{3-3-1}.
If $q=p+\sigma$, the coefficient of $L$ becomes
$$
\widetilde{C}-(2\|u\|_{L^\infty(B_1)})^{q-p+1}\mathcal{H}(3\alpha )^{q-p},
$$
and then we may take a universal $\alpha\in(0,1)$ small enough that
$$
\frac{\widetilde{C}}{2}\le\widetilde{C}-(2\|u\|_{L^\infty(B_1)})^{q-p+1}\mathcal{H}(3\alpha )^{q-p}.
$$
Finally, by choosing $L$ large enough we may arrive at a contradiction.
\end{proof}

Based on the H\"{o}lder continuity of solutions above, Lemma \ref{lem3-2}, we could further justify the Lipschitz regularity (Theorem \ref{thm2}) for Eq. \eqref{main}, whose proof is stated as follows.

\begin{proof}[\textbf{Proof of Theorem \ref{thm2}}]
Lemma \ref{lem3-3} implies that $u$ is locally $\alpha$-H\"{o}lder continuous in $B_1$.
 For the sake of convenience, we assume $u\in C^{0,\alpha}(\overline{B}_1)$ with $\alpha<1$.
The argument is analogous to that of Lemma \ref{lem3-1}, with the difference that the H\"{o}lder continuity of $u$ is applied instead of the boundedness of $u$.

Let us use the same notation as  in the proof of Lemma \ref{lem3-1}. Corresponding to the display \eqref{3-1-2}, by H\"{o}lder continuity of $u$, we have
\begin{align*}
0&<\Psi(\overline{x},\overline{y})\le u(\overline{x})-u(\overline{y})-L\omega(|\overline{x}-\overline{y}|)-\psi(\overline{y}) \nonumber\\
&\le[u]_\alpha|\overline{x}-\overline{y}|^\alpha-L\omega(|\overline{x}-\overline{y}|)-m\left[\left(|\overline{y}|^2-\frac{1}{4}\right)_+\right]^2.
\end{align*}
By taking
$$
m>\frac{512}{25}[u]_\alpha
\quad\text{and}\quad
L>16\biggl(1-\frac{1}{4}\left(\frac{1}{8}\right)^\beta\biggr)^{-1}[u]_\alpha,
$$
we conclude that $\overline{x}\neq\overline{y}$, $|\overline{x}-\overline{y}|<\frac{1}{8}$ and $(\overline{x},\overline{y})\in B_\frac{7}{8}\times B_\frac{7}{8}\subset B_1\times B_1$, together with
\begin{equation}
\label{3-4-1}
|\overline{x}-\overline{y}|^{1-\alpha}\le\frac{2[u]_\alpha}{L}.
\end{equation}

Following the argument in Case 1 of the proof of Lemma \ref{lem3-1}, except letting $\xi=0$ this time, the inequality \eqref{3-1-4} becomes
$$
\mathcal{I}_\sigma(v,\overline{x})-\mathcal{I}_\sigma(w,\overline{y})\ge -C(p)(\mathcal{M}+\|f\|_{L^\infty(B_1)})-C(q)\mathcal{H}L^{q-p}
$$
and further \eqref{3-1-6} (as $\delta\rightarrow0$) turns into
\begin{align*}
&-C(\|f\|_{L^\infty(B_1)}+\|u\|_{L^\infty(B_1)}+\|u\|_{L^1_\sigma}+\mathcal{M})-C\mathcal{H}L^{q-p}\\
&\qquad\le 2LI_K[\mathcal{D}](\omega,\overline{a})
\le -CL|\overline{a}|^{\beta-1+\frac{N-1}{2}\beta+(\beta+1)(2-\sigma)}.
\end{align*}
Note that $|\overline{a}|<\frac{1}{8}$ and $\iota=\beta-1+\frac{N-1}{2}\beta+(\beta+1)(2-\sigma)<0$. Thus, we have
 \begin{equation}
\label{3-4-2}
\begin{split}
CL&\le C(\|f\|_{L^\infty(B_1)}+\|u\|_{L^\infty(B_1)}+\|u\|_{L^1_\sigma}+\mathcal{M})+C\mathcal{H}L^{q-p}|\overline{x}-\overline{y}|^{-\iota} \\
&\le C(\|f\|_{L^\infty(B_1)}+\|u\|_{L^\infty(B_1)}+\|u\|_{L^1_\sigma}+\mathcal{M})+C\mathcal{H}(2[u]_\alpha)^\frac{-\iota}{1-\alpha}
L^{q-p+\frac{\iota}{1-\alpha}},
\end{split}
\end{equation}
where we applied \eqref{3-4-1}. We observe that
\begin{align*}
q-p+\frac{\iota}{1-\alpha}&\le \sigma+\frac{1-\sigma+\beta\left(3+\frac{N-1}{2}-\sigma\right)}{1-\alpha}\\
&=1+\frac{\beta\left(3+\frac{N-1}{2}-\sigma\right)-\alpha(\sigma-1)}{1-\alpha}.
\end{align*}
To get a contradiction from \eqref{3-4-2}, we require that $q-p+\frac{\iota}{1-\alpha}<1$, or equivalently,
$$
\beta\left(3+\frac{N-1}{2}-\sigma\right)<\alpha(\sigma-1).
$$
Thus we may select $0<\beta<\frac{\alpha(\sigma-1)}{3+\frac{N-1}{2}-\sigma}$.
Once we make this choice, we may choose $L>0$ sufficiently large in the inequality \eqref{3-4-2} to get a contradiction.
This proves the Lipschitz continuity.
\end{proof}

\begin{remark}
In the scenario $p<q\le p+1$, we could also deduce the Lipschitz regularity for the solutions to the gradient perturbation equation \eqref{main2} by means of the bootstrap argument as the proof of Theorem \ref{thm2}. However, the H\"{o}lder continuity of solutions to \eqref{main2} (Lemma \ref{lem3-2}) suffices to provide the compactness results needed in Lemma \ref{lem4-1} below.
\end{remark}

\section{Gradient H\"{o}lder regularity of solutions}
\label{sec4}

In this section, we are going to demonstrate the gradient H\"{o}lder continuity of viscosity solutions to \eqref{main} by a perturbation argument. The core of this idea is to find a solution with known $C^{1,\alpha}$-regularity to approximate the solution of \eqref{main}, so that this property can be transferred to the solution under consideration. Therefore, through the flatness improvement argument, we shall first show the following approximation lemma that implies the nonlocal equation \eqref{main2} can be uniformly close to a (local) uniformly elliptic equation when $\sigma$ approaches to two.

\begin{lemma}
\label{lem4-1}
Let the conditions $(A_1)$--$(A_4)$ with $0\le q\le p+1$ hold true. Let $u\in C(\overline{B}_1)$ be a normalized viscosity solution to \eqref{main2}.
Given $\mu,\varepsilon>0$, there exists $\kappa>0$ depending on $N$, $\lambda$, $\Lambda$, $\varepsilon$, $p$ and $\mu$ 
such that if
$$
|\sigma-2|+\|f\|_{L^\infty(B_1)}+\mathcal{M}+\mathcal{H}\bigl(1+|\xi|^{(q-p)_+}\bigr)\leq \kappa
$$
and
$$
|u(x)|\le\mu(1+|x|^{1+\overline{\alpha}}), \quad  x\in\mathbb{R}^N,
$$
with some $\overline{\alpha}\in(0,1)$, then there exists an $F$-harmonic function $h\in C^{1,\overline{\alpha}}_{\rm loc}(B_1)$ (i.e., $h$ is a solution to $F(D^2h)=0$ in the viscosity sense) satisfying
$$
\|u-h\|_{L^\infty(B_{\frac12})}\leq \varepsilon.
$$
\end{lemma}

\begin{proof}
For a contradiction, assume that there exist $\mu_0,\varepsilon_0>0$ and sequences of $(\sigma_j)$, $(f_j)$, $(u_j)$, $(H_j)$, $(\xi_j)$ such that
\begin{itemize}
  \item[(i)]
  $u_j\in C(B_1)$ with $\|u_j\|_{L^\infty(B_1)}\leq 1$ and $|u_j(x)|\le\mu_0(1+|x|^{1+\overline{\alpha}})$ is a solution to
      \begin{equation}
      \label{4-1-1}
      -|Du_j+\xi_j|^p\mathcal{I}_{\sigma_j}(u_j,x)+H_j(x,Du_j+\xi_j)=f_j(x)  \quad \text{in } B_1,
      \end{equation}
      where the operator $\mathcal{I}_{\sigma_j}$ satisfies $(A_4)$;
  \item[(ii)] the Hamiltonian term fulfills
  \begin{equation}\label{4-1-2}
  |H_j(x,\xi)|\le \mathcal{M}_j+\mathcal{H}_j|\xi|^{q-p};
  \end{equation}
  \item[(iii)] and
  \begin{equation}
  \label{4-1-3}
|\sigma_j-2|+\|f_j\|_{L^\infty(B_1)}+\mathcal{M}_j+\mathcal{H}_j\bigl(1+|\xi_j|^{(q-p)_+}\bigr)\leq \frac{1}{j}.
\end{equation}

\end{itemize}
However, we have
$$
\|u_k-h\|_{L^\infty(B_{\frac12})}>\varepsilon_0
$$
for any $h(x)\in C^{1,\overline{\alpha}}_{\rm loc}(B_1)$. With the help of $\sigma_j\rightarrow2$ and the condition ($A_4$), we have $\mathcal{I}_{\sigma_j}\rightarrow F$, where $F$ is a uniformly $(\lambda,\Lambda)$-elliptic operator. Moreover, based on $\mathcal{H}_j\left(1+|\xi_j|^{(q-p)_+}\right)\leq \frac{1}{j}$, Lemma \ref{lem3-1} and Lemma \ref{lem3-2}, we can find a continuous function $v$ such that $u_j\rightarrow v$ locally uniformly in $B_1$ by compactness. In particular, we have
$v\in C(B_{\frac34})$, $\|v\|_{L^\infty(B_{\frac34})}\leq1$ and
\begin{equation}
\label{4-1-4}
\|v-h\|_{L^\infty(B_{\frac12}}>\varepsilon_0.
\end{equation}

Next, we are ready to verify that $v$ is a viscosity solution to
\begin{equation}
\label{4-1-5}
-F(D^2u)=0   \quad \text{in }  B_{\frac34}.
\end{equation}
We shall only prove that $v$ is a viscosity supersolution, since the argument for $v$ is a subsolution is similar.
Let $\varphi$ be a smooth function such that $v-\varphi$ has a local minimum in $B_1$ at $\tilde{x}$.
Without loss of generality we may assume that $|\tilde{x}|=v(0)=\varphi(0)=0$ and that $\varphi$ is a quadratic polynomial, i.e.,
$$
\varphi(x)=\frac{1}{2}Ax\cdot x+\eta\cdot x,
$$
where $\eta\in\mathbb{R}^N$ and $A$ is a symmetric matrix and the dot means the standard inner product.
Since $u_j\rightarrow v$ locally uniformly in $B_1$, there exists a point $x_k$ in a small neighborhood of the origin and a quadratic polynomial
$$
\varphi_j(x)=\frac{1}{2}A(x-x_j)\cdot(x-x_j)+\eta\cdot(x-x_j)+u_j(x_j)
$$
touching $u_j$ from below at $x_j$. By $u_j$ a viscosity solution to \eqref{4-1-1}, we have
\begin{equation}
\label{4-1-6}
-|\eta+\xi_j|^p\mathcal{I}_{\sigma_j}(u_j,x)+H_j(x,\eta+\xi_j)\ge f_j(x_j).
\end{equation}

In the scenario that $\{\xi_j\}$ is unbounded, we assume $|\xi_j|\rightarrow\infty$ up to a subsequence. Using \eqref{4-1-2}, \eqref{4-1-3}, we obtain
\begin{align*}
\frac{|H_j(x_j,\eta+\xi_j)|}{|\eta+\xi_j|^p}
&\le |\xi_j|^{q-p}\frac{\mathcal{M}_j|\xi_j|^{-q}+\mathcal{H}_j|\hat{\xi}_j+|\xi_j|^{-1}\eta|^q}{|\hat{\xi}_j+|\xi_j|^{-1}\eta|^p}\\
&\le 2^{q+p}\frac{|\xi_j|^{q-p}}{j(|\xi_j|^{(q-p)_+}+1)}\rightarrow0
\end{align*}
and
$$
\frac{f_j(x_j)}{|\eta+\xi_j|^p}\rightarrow0
$$
as $j\rightarrow\infty$. Here $\hat{x}=\frac{x}{|x|}$. Hence passing to the limit in \eqref{4-1-6}, we get
$$
-F(A)=-\lim_{k\rightarrow\infty}\mathcal{I}^\delta_{\sigma_j}(u_j,\varphi_j,x_j)\ge 0.
$$

If $(xi_j)$ is a bounded sequence, we may assume that $\xi_j\rightarrow\overline{\xi}$ up to a subsequence.
We first consider the case $|\eta+\overline{\xi}|>0$. Applying \eqref{4-1-2}, \eqref{4-1-3},
we conclude that
\begin{equation*}
\frac{|f_j(x_j)|}{|\eta+\xi_j|^p}\le\frac{2^p}{j|\eta+\overline{\xi}|^p}\rightarrow0
\end{equation*}
and
$$
\frac{|H_j(x_j,\eta+\xi_j)|}{|\eta+\xi_j|^p}\le \frac{\mathcal{M}_j+\mathcal{H}_j|\xi_j+\eta|^q}{|\xi_j+\eta|^p}\le \frac{1+|\xi_j+\eta|^q}{j|\eta+\xi_j|^p(|\xi_j|^{(q-p)_+}+1)}\rightarrow0
$$
as $j\rightarrow\infty$. At this moment, passing to the limit in \eqref{4-1-6}, we also obtain
$$
-F(A)=-\lim_{k\rightarrow\infty}\mathcal{I}^\delta_{\sigma_j}(u_j,\varphi_j,x_j)\ge 0.
$$

In the sequel, we focus on the situation $|\eta+\overline{\xi}|=0$, for which we show that $F(A)\le0$.
For a contradiction, assume that
\begin{equation}
\label{4-1-7}
F(A)>0.
\end{equation}
From this and the uniform ellipticity of $F(\cdot)$, it follows that there is at least one positive eigenvalue related to the matrix $A$.
Consider the orthogonal sum $\mathbb{R}^n=T\oplus Q$ with $T$ 
denoting the invariant space consisting of the eigenvectors corresponding to positive eigenvalues.

\medskip

\textbf{Case 1.} $b=-\overline{\eta}\neq0$. Let $\gamma>0$ and
$$
\varphi_\gamma(x)=\varphi(x)+\gamma|P_T(x)|=\frac{1}{2}Ax\cdot x+\eta\cdot x+\gamma|P_T(x)|,
$$
where $P_T$ stands for the orthogonal projection over $T$.
Since $u_j\rightarrow v$ locally uniformly and $\varphi$ touches $v$ at 0 from below, it follows that $\varphi_\gamma$ touches $u_j$ from below at some $x^\gamma_j$ in a neighbourhood of 0 for $\gamma$ small enough, and moreover $x^\gamma_j\rightarrow\overline{x}$ for some $\overline{x}\in B_{\frac34}$ as $j\rightarrow\infty$ (up to a subsequence).

If $P_T(x^\gamma_j)=0$, we observe that
$$
\phi_\gamma(x)=\frac{1}{2}Ax\cdot x+\eta\cdot x+\gamma e\cdot P_T(x)
$$
touches $u_j$ from below at $x^\gamma_j$ for every $e\in \mathbb{S}^{N-1}$ (i.e., $|e|=1$). This implies a viscosity inequality
\begin{equation}
\label{4-1-8}
\begin{split}
&-|\eta+Ax^\gamma_j+\gamma P_T(e)+\xi_j|^p\mathcal{I}^\delta_{\sigma_j}(u_j,\phi_\gamma,x^\gamma_j)\\
&\qquad+H_j(x_j^\gamma,\eta+Ax^\gamma_j+\gamma P_T(e)+\xi_j)\ge f_j(x^\gamma_j),
\end{split}
\end{equation}
where $D(e\cdot P_T(x))=P_T(e)$. If $A\overline{x}=0$, then we can see $|\eta+\xi_j+Ax^\gamma_j|\rightarrow0$, so we get
\begin{equation*}
2\gamma>|\eta+Ax^\gamma_j+\gamma P_T(e)+\xi_j|>\frac{\gamma}{2}
\end{equation*}
for large enough $j$, through selecting $e\in T\cap \mathbb{S}^{N-1}$ such that $P_T(e)=e$.
On the other hand, if $A\overline{x}\neq0$ and $T\equiv \mathbb{R}^N$, we choose $e\in \mathbb{S}^{N-1}$ fulfilling
\begin{align*}
2|A\overline{x}+\gamma e|
&>|\eta+Ax^\gamma_j+\gamma P_T(e)+\xi_j|\\
&\ge\frac{1}{2}|A\overline{x}+\gamma e|-\frac{1}{4}|A\overline{x}+\gamma e|
=\frac{1}{4}|A\overline{x}+\gamma e|>0,
\end{align*}
where we applied $|\eta+\xi_j|\rightarrow0$.
Besides, if $A\overline{x}\neq0$ and $T\neq\mathbb{R}^N$, we choose $e\in Q\cap \mathbb{S}^{N-1}$ such that
$$
2|A\overline{x}|>|\eta+Ax^\gamma_j+\gamma P_T(e)+\xi_j|\ge\frac{1}{2}|A\overline{x}|-\frac{1}{4}|A\overline{x}|=\frac{1}{4}|A\overline{x}|>0,
$$
where we applied the facts that $P_T(e)=0$ and $|\eta+\xi_j|\rightarrow0$. Thereby, employing \eqref{4-1-2}, \eqref{4-1-3} we have
\begin{align*}
\frac{|H_j(x^\gamma_j,\eta+Ax^\gamma_j+\gamma P_T(e)+\xi_j)|}{|\eta+Ax^\gamma_j+\gamma P_T(e)+\xi_j|^p}&\le \frac{\mathcal{M}_j+\mathcal{H}_j|Ax^\gamma_j+\gamma P_T(e)+\xi_j+\eta|^q}{a^p}\\
&\le \frac{1+b^q}{a^p}(\mathcal{M}_j+\mathcal{H}_j)\le\frac{1+b^q}{a^p}\frac{1}{j(|\xi_j|^{(q-p)_+}+1)}\rightarrow0
\end{align*}
and
$$
\frac{|f_j(x^\gamma_j)|}{|\eta+Ax^\gamma_j+\gamma P_T(e)+\xi_j|^p}\le\frac{1}{ja^p}\rightarrow0
$$
as $j\rightarrow\infty$, where $a,b\in\{(\gamma/2,2\gamma),(|A\overline{x}+\gamma e|/4,2|A\overline{x}+\gamma e|),(|A\overline{x}|/4,2|A\overline{x}|)\}$. As a consequence, by taking the limit in \eqref{4-1-8}, we conclude that $-F(A)\ge0$, contradicting \eqref{4-1-7}.

Finally, we address the case when $P_T(x^\gamma_j)\neq0$. We observe that the function $x\rightarrow|P_T(x)|$ is convex and smooth near $x^\gamma_j$. Moreover, we have
\begin{equation}
\label{4-1-9}
|P_T(x)|D(|P_T(x)|)=P_T(x) \ \text{ and } \ D^2(|P_T(x)|) \ \text{is nonnegative definite}.
\end{equation}
Hence we obtain the following viscosity inequality
\begin{align*}
-|\eta+Ax^\gamma_j+\gamma e^\gamma_j+\xi_j|^p\mathcal{I}^\delta_{\sigma_j}(u_j,\varphi_\gamma,x^\gamma_j)+H_j(x^\gamma_j,\eta+Ax^\gamma_j+\gamma e^\gamma_j+\xi_j)\ge f_j(x^\gamma_j)
\end{align*}
with
$$
e^\gamma_j=\frac{P_T(x^\gamma_j)}{|P_T(x^\gamma_j)|}.
$$
As in the case when $P_T(x^\gamma_j)=0$, we may considere the scenarios that $A\overline{x}=0$ and $A\overline{x}\neq0$, respectively, and obtain $-F(A+D^2|P_T(\overline{x})|)\ge0$.
By \eqref{4-1-9} and the ellipticity on $F$, we have $F(A)\le0$. This is a contradiction with \eqref{4-1-7}.

\medskip

\textbf{Case 2.} $p=\eta=0$. For this, the procedures is easier. Since $\frac{1}{2}Ax\cdot x$ touches $\overline{u}$ from below at 0 and $u_j\rightarrow\overline{u}$ locally uniformly, it follows that
$$
\hat{\varphi}_\gamma(x)=\frac{1}{2}Ax\cdot x+\gamma|P_T(x)|
$$
touches $u_j$ from below at some $\hat{x}_j$ in a small neighborhood of 0 for $j$ large enough.
As in  Case 1, we examine the cases $|P_T(\hat{x}_j)|=0$ and $|P_T(\hat{x}_j)|>0$ separately, and obtain estimates for
$$|\eta_j+A\hat{x}_j+e|
\quad\text{and}\quad
|\eta_j+A\hat{x}_j+\gamma \hat{e}|,
\quad \hat{e}=\frac{P_T(\hat{x}_j)}{|P_T(\hat{x}_j)|},\ (P_T(\hat{x}_j)\neq0).
$$
We conclude that $F(A)\leq0$, which contradicts \eqref{4-1-7}.

At this point, we have demonstrated that $v$ is a viscosity supersolution to \eqref{4-1-5} and similarly we may verify that it is a subsolution. It is well known in \cite[Chapter 5]{CC95} that a solution of \eqref{4-1-5} is $C^{1,\overline{\alpha}}_{\rm loc}$-regular for some $\overline{\alpha}\in (0,1)$. Hence we choose $h=v$ and obtain a contradiction with \eqref{4-1-4}. This completes the proof.
\end{proof}

Next, we would like to check the error on the solutions to \eqref{main} and a linear function, which will implies the $C^{1,\alpha}$ properties of solutions by the well-known theory.

\begin{lemma}
\label{lem4-2}
Let the hypotheses $(A_1)$--$(A_4)$ be fulfilled with $0\le q\le p+1$. Suppose that $u$ is a normalized viscosity solution to \eqref{main}. Given $\mu>0$, there exists $\nu>0$,
depending on $N$, $\lambda$, $\Lambda$, $p$, $q$, $\alpha$ and $\mu$,  such that if
$$
|\sigma-2|+\|f\|_{L^\infty(B_1)}+\mathcal{M}+\mathcal{H}\leq \nu
$$
and
$$
|u(x)|\le\mu(1+|x|^{1+\overline{\alpha}}), \quad x\in\mathbb{R}^N
$$
with $\overline{\alpha}$ from Lemma \ref{lem4-1}, then there  exists $\rho\in\left(0,\frac{1}{2}\right)$, depending on $N$, $\lambda$, $\Lambda$ and $\alpha$, and a sequence $(l_j)$ of affine functions $l_j(x)=a_j+b_j\cdot x$ for which
\begin{equation}
\label{4-2-0}
\|u-l_j\|_{L^\infty(B_{\rho^j})}\leq \rho^{j(1+\alpha)}
\end{equation}
with
\begin{equation}
\label{4-2-0-1}
|a_{j+1}-a_j|+\rho^j|b_{j+1}-b_j|\leq C\rho^{j(1+\alpha)}
\end{equation}
for any
$\alpha\in(0,\overline{\alpha})\cap \bigl(0,\tfrac{\sigma-1}{1+p}\bigr]$.
Here the constant $C\ge1$ depends only on $N$, $\lambda$ and $\Lambda$.
\end{lemma}

\begin{proof}
\textbf{Step 1.} We first find an affine function $l$ and a number $\rho\in\left(0,\frac{1}{2}\right)$ satisfying
\begin{equation}
\label{4-2-1}
 \sup_{x\in B_\rho}|u(x)-l(x)|\le\rho^{1+\alpha}.
\end{equation}
By Lemma \ref{lem4-1}, there exists an $F$-harmonic function $h\in C^{1,\overline{\alpha}}(B_{\frac34})$ such that
$$
\|u-h\|_{L^\infty(B_{\frac34})}\le\varepsilon
$$
with $\varepsilon>0$ to be selected later. The existence of such an $F$-harmonic function is ensured by Lemma \ref{lem4-1}, provided $\nu>0$ is small enough.

From the regularity theory in \cite{CC95}, it follows that
$$
\sup_{x\in B_\rho}|h(x)-(h(0)+Dh(0)\cdot x)|\leq C\rho^{1+\overline{\alpha}},
\quad \rho\in\left(0,\tfrac34\right),
$$
with $|h(0)|+|Dh(0)|\leq C$.
Here the constants $C$ and $\overline{\alpha}\in(0,1)$ depend only on $N,\lambda,\Lambda$. Set
$$
l(x)=a_1+b_1\cdot x=h(0)+Dh(0)\cdot x.
$$
Then
\begin{align*}
\sup_{x\in B_\rho}|u(x)-l(x)|\leq \sup_{x\in B_\rho}|u(x)-h(x)|+\sup_{x\in B_\rho}|h(x)-l(x)|<\varepsilon+C\rho^{1+\overline{\alpha}}.
\end{align*}
Since $0<\alpha<\overline{\alpha}$, we take $0<\rho\ll1$ such that
$$
\rho\leq(2C)^{-\frac{1}{\overline{\alpha}-\alpha}} \quad\text{and}\quad  \varepsilon=\frac{1}{2}\rho^{1+\alpha}
$$
to obtain \eqref{4-2-1}. Once we fix the value of $\varepsilon$ here, the quantity $\kappa$ in  Lemma \ref{lem4-1} is determined accordingly.
Let  $\nu>0$ be small enough that
\begin{equation*}
\nu\biggl(\left(C+\frac{C}{1-\rho^\alpha}\right)^{q-p}+1\biggr)\le \frac{\kappa}{4}
\end{equation*}
with the $\kappa$ coming from Lemma \ref{lem4-1}. 
This leads to the claim.

\medskip

\textbf{Step 2}. Proceed by induction. For $k=1$, the claim has been proved in Step 1. Suppose that the claim is true for $k=1,2,\dots,j$.
We will show the claim for $k=j+1$. Let
$$
u_j(x)=\frac{u(\rho^j x)-l_j(\rho^j x)}{\rho^{j(1+\alpha)}}.
$$
Then
\begin{equation}
\label{4-2-2}
-|Du_j+\xi_j|^p\overline{\mathcal{I}}(u_j,x)+\overline{H}(x,Du_j+\xi_j)=\overline{f}(x)  \quad\text{in }  B_1,
\end{equation}
where the nonlocal operator $\overline{\mathcal{I}}$ carries the same uniform ellipticity condition as the operator $\mathcal{I}$ in \eqref{main} (more detains on this could be seen in \cite{CS11}).
Here
$$
\xi_j=\rho^{-j\alpha}b_j, \quad \overline{f}(x)=\rho^{j(\sigma-1-\alpha(1+p))}f(\rho^jx)
$$
and
$$
\overline{H}(x,Du_j+\xi_j)=\rho^{j(\sigma-1-\alpha(1+p))}H(\rho^jx,\rho^{j\alpha}(Du_j+\xi_j)).
$$
By $\alpha\le\frac{\sigma-1}{p+1}$ and \eqref{4-2-0} together with the growth condition \eqref{2-1} on $H(\cdot)$, we obtain
$$
\|\overline{f}\|_{L^\infty(B_1)}\le \rho^{j(\sigma-1-\alpha(1+p))}\|f\|_{L^\infty(B_1)}\le\nu
$$
and
\begin{align*}
|\overline{H}(x,Du_j+\xi_j)|&\le\rho^{j(\sigma-1-\alpha(1+p))}\left(\mathcal{M}+\mathcal{H}\rho^{jq\alpha}|Du_j+\xi_j|^q\right)\\
&=\mathcal{M}_j+\mathcal{H}_j|Du_j+\xi_j|^q
\end{align*}
with
$$
\mathcal{M}_j=\rho^{j(\sigma-1-\alpha(1+p))}\mathcal{M} \quad\text{and}\quad \mathcal{H}_j=\rho^{j(\sigma-1-\alpha(1+p-q))}\mathcal{H}.
$$
By $\alpha\le\frac{\sigma-1}{p+1}$ and \eqref{4-2-0} again, we obtain
$$
\mathcal{M}_j\le \mathcal{M}\le \nu \quad\text{and}\quad  \mathcal{H}_j\le\mathcal{H}\le \nu.
$$

Next we analyze the quantity $\overline{H}_j(|\xi_j|^{(q-p)_+}+1)$. For $q\in[0,p]$, we have
$$
\overline{H}_j(|\xi_j|^{(q-p)_+}+1)\le2\mathcal{H}\le2\nu.
$$
Through the induction hypothesis, we have $|b_{j+1}-b_j|\le C\rho^{j\alpha}$, so that we obtain
$$
|b_j|\le |b_1|+\sum_{k=1}^{j-1}|b_{k+1}-b_k|\le C+C\sum_{k=1}^{j-1}\rho^{k\alpha}\le C+\frac{C}{1-\rho^\alpha}.
$$
Thus, for $q\in(p,p+1]$, it follows that
\begin{align*}
\overline{H}_j(|\xi_j|^{q-p}+1)&=\rho^{j(\sigma-1-\alpha(1+p-q))}\mathcal{H}\bigl(\rho^{-j\alpha(q-p)}|b_j|^{q-p}+1\bigr)\\
&\le\rho^{j(\sigma-1-\alpha}\mathcal{H}\left(|b_j|^{q-p}+1\right)\\
&\le\nu\biggl(\left(C+\frac{C}{1-\rho^\alpha}\right)^{q-p}+1\biggr).
\end{align*}
In summary, for $q\in[0,p+1]$, we have
$$
\overline{H}_j(|\xi_j|^{(q-p)_+}+1)\le\nu\biggl(\left(C+\frac{C}{1-\rho^\alpha}\right)^{q-p}+1\biggr).
$$
At this moment,  equation \eqref{4-2-2} satisfies the smallness conditions in Lemma \ref{lem4-1}. It remains to justify that
\begin{equation}
\label{4-2-3}
|u_j(x)|\le1+|x|^{1+\overline{\alpha}} \quad\text{for } x\in\mathbb{R}^N.
\end{equation}
If  \eqref{4-2-3} holds, we may apply the conclusion from Step 1 to deduce that
$$
\sup_{B_\rho}|u_j(x)-\tilde{l}(x)|\le \rho^{1+\alpha},
$$
where $\tilde{l}$ is an affine function of the form
$$
\tilde{l}(x)=\tilde{a}+\tilde{b}\cdot x=\tilde{h}(0)+D\tilde{h}(0)\cdot x
$$
with $|\tilde{a}|+|\tilde{b}|\le C(N,\lambda,\Lambda)$. Scaling back, we infer that
$$
\sup_{B_{\rho^{j+1}}}|u(x)-l_{j+1}(x)|\le \rho^{(j+1)(1+\alpha)},
$$
where
$$
l_{j+1}(x)=a_{j+1}+b_{j+1}\cdot x=(a_j+\rho^{j(1+\alpha)}\tilde{a})+(b_j+\rho^{j\alpha}\tilde{b}\cdot x).
$$
Then
$$
|a_{j+1}-a_j|\le C\rho^{j(1+\alpha)} \quad \text{and}\quad |b_{j+1}-b_j|\le C\rho^{j\alpha},
$$
as desired.

Finally, we verify \eqref{4-2-3} via induction to completes the proof.
The argument is analogous to that in the proof of \cite[Lemma 3.3]{FRZ25}, but we give the details for the sake of completeness.
For $k=0$, take $u_0=u$. Suppose that \eqref{4-2-3} holds for $k=0,1,2,\dots,j$. Next, we proceed to prove this for $k=j+1$. Notice that
$$
u_{j+1}(x)=\rho^{-(1+\alpha)}\left(\frac{u(\rho^j(\rho x))-l_j(\rho^j(\rho x))}{\rho^{j(1+\alpha)}}-\tilde{l}(\rho x)\right)
=\frac{u_j(\rho x)-\tilde{l}(\rho x)}{\rho^{1+\alpha}}.
$$
If $\rho|x|>\frac{1}{2}$, we choose
$$
\rho\le \biggl(\frac{1}{10(1+C)}\biggr)^\frac{1}{\overline{\alpha}-\alpha}
$$
to have
\begin{align*}
|u_{j+1}(x)|&\le \rho^{-(1+\alpha)}\bigl((1+|\rho x|^{1+\overline{\alpha}})+C(1+|\rho x|)\bigr)\\
&\le\rho^{\overline{\alpha}-\alpha}(5+6C)|x|^{1+\overline{\alpha}}
\le|x|^{1+\overline{\alpha}}.
\end{align*}
On the other hand, when $\rho|x|\le\frac{1}{2}$, it holds that
\begin{align*}
|u_{j+1}(x)|&\le \rho^{-(1+\alpha)}(|u_j(\rho x)-h(\rho x)|+|h(\rho x)-\tilde{l}(\rho x)|)\\
&\le\rho^{-(1+\alpha)}\left(\frac{\rho^{1+\alpha}}{2}+C\rho^{1+\overline{\alpha}}|x|^{1+\overline{\alpha}}\right)\\
&\le\frac{1}{2}+C\rho^{\overline{\alpha}-\alpha}|x|^{1+\overline{\alpha}}
\le1+|x|^{1+\overline{\alpha}}
\end{align*}
with an $F$-harmonic function $h$ from Lemma \ref{lem4-1}. Finally, we let
$$
\rho=\frac{1}{2}\min\biggl\{\biggl(\frac{1}{2C}\biggr)^\frac{1}{\overline{\alpha}-\alpha},\biggl(\frac{1}{10(1+C)}\biggr)^\frac{1}{\overline{\alpha}-\alpha}\biggr\}.
$$
This completes the proof.
\end{proof}

In the end of this section, we summarize the proof of Theorem \ref{thm1}.

\begin{proof}[\textbf{Proof of Theorem \ref{thm1}}]
We apply the scaling properties of \eqref{main} to reduce the problem to a smallness regime, and then we apply Lemma \ref{lem4-2}. That is to say, we have to justify the smallness hypotheses that
\begin{equation}
\label{4-3-1}
\|u\|_{L^\infty(B_1)}\le1 \quad\text{and}\quad  \|f\|_{L^\infty(B_1)}+\mathcal{M}+\mathcal{H}\le \varepsilon
\end{equation}
with $0<\varepsilon\ll1$ for \eqref{main}. Let $v:\mathbb{R}^N\rightarrow\mathbb{R}$,
$$
v(x)=\frac{u(x_0+rx)}{K},
$$
where $0<r<1$ such that $B_r(x_0)\subset B_1$ and $K\ge1$ is a number to be determined later. If $u$ is a viscosity solution to \eqref{main} in $B_1$, then $v$ is a solution to
\begin{equation}
\label{4-3-2}
-|Dv|^p\overline{\mathcal{I}}(v,x)+\overline{H}(x,Dv)=\overline{f}(x) \quad\text{in } B_1,
\end{equation}
where $\overline{\mathcal{I}}$ is a uniformly elliptic nonlocal operator of the same ellipticity condition as $\mathcal{I}$ in \eqref{main},
$$
\overline{f}(x)=\frac{r^{p+\sigma}}{K^{p+1}}f(x_0+rx)
$$
and
$$
\overline{H}(x,\xi)=\frac{r^{p+\sigma}}{K^{p+1}}H\left(x_0+rx,\frac{K}{r}\xi\right)
$$
with
$$
|\overline{H}(x,\xi)|\le\frac{r^{p+\sigma}}{K^{p+1}}\left(\mathcal{M}+\biggl(\frac{K}{r}\right)^q\mathcal{H}|\xi|^q\biggr)=\overline{\mathcal{M}}
+\overline{\mathcal{H}}|\xi|^q.
$$
If $q<p+1$, by choosing
$$
K=1+\|u\|_{L^\infty(B_1)}+\left(\frac{\mathcal{M}+\|f\|_{L^\infty(B_1)}}{\varepsilon}\right)^\frac{1}{1+p}+\left(\frac{\mathcal{H}}{\varepsilon}
\right)^\frac{1}{1+p-q},
$$
we conclude that
$$
\|v\|_{L^\infty(B_1)}\le1, \quad \|\overline{f}\|_{L^\infty(B_1)}\le \frac{r^{p+\sigma}}{K^{p+1}}\|f\|_{L^\infty(B_1)}\le\varepsilon
$$
as well as
$$
\overline{\mathcal{M}}=\frac{r^{p+\sigma}}{K^{p+1}}\mathcal{M}\le\varepsilon
\quad\text{and}\quad
\overline{\mathcal{H}}=\frac{r^{p+\sigma-q}}{K^{p+1-q}}\mathcal{H}\le\varepsilon
$$
for all $r\in(0,1)$. For the case $q=p+1$, let
$$
K=1+\|u\|_{L^\infty(B_1)} \quad\text{and}\quad  r=\min\biggl\{\frac{1}{2},\left(\frac{\varepsilon}{4\|f\|_{L^\infty(B_1)}+4\mathcal{M}}\right)^\frac{1}{p+\sigma},\left(4\mathcal{H}\right)
^\frac{1}{\sigma-1}\biggr\}.
$$
Then we obtain
$$
\|v\|_{L^\infty(B_1)}\le1 \quad\text{and}\quad  \|\overline{f}\|_{L^\infty(B_1)}+\overline{\mathcal{M}}+\overline{\mathcal{H}}\le\varepsilon.
$$
In conclusion, $v$ is a solution to \eqref{4-3-2} in the same class as \eqref{main} with the smallness assumptions in \eqref{4-3-1}.

At this point, the assumptions in Lemma \ref{lem4-2} are fulfilled so that we could utilize this lemma to arrive at Theorem \ref{thm1}. Observe that the Cauchy sequences $(a_j)$ and $(b_j)$ converge to the limits $a_\infty$ and $b_\infty$, respectively. By employing the discrete iteration expression \eqref{4-2-0}, we obtain
\begin{equation}
\label{4-3-3}
\sup_{B_\varrho}|u(x)-l_\infty(x)|\le C\varrho^{1+\alpha}
\end{equation}
for any $\varrho\in(0,\rho]$, where $l_\infty(x)=a_\infty+b_\infty\cdot x$ and $C\ge1$ is a universal constant.
This implies the $C^{1,\alpha}$-regularity of the viscosity solution $u$ to \eqref{main}.
The estimate on \eqref{4-3-3} is very standard, the details of which can be referred to e.g. \cite{FRZ21,APS24}.
\end{proof}

\section{$C^1$-regularity of solutions}
\label{sec5}

In this section, we set out to weaken the law of degeneracy of the equation to investigate the borderline regularity. To be more precise, for the equation
\begin{equation}
\label{main3}
\gamma(|Du|)\mathcal{I}_\sigma(u,x)+H(x,Du)=f(x) \quad\text{in } B_1,
\end{equation}
the objective is to identify appropriate conditions on $\gamma(\cdot)$ to ensure that solutions of \eqref{main3} are continuously differentiable.

We first present concepts of non-collapsing sets and shored-up sequences, as well as two known lemmas from \cite{APPT22} as follows.

\begin{definition}
\label{def5-1}
A set $\Gamma$ of moduli of continuity defined on an interval $I\subset(0,\infty)$ is said to be non-collapsing, whenever for any sequence $(f_j)\subset\Gamma$ and any sequence $(a_j)\subset I$, we have that $f_j(a_j)\rightarrow$ implies that $a_j\rightarrow0$.
\end{definition}

\begin{definition}
\label{def5-2}
A sequence of moduli of continuity $(\gamma_j)$ is said to be shored-up, if there exists a sequence of positive numbers $\{a_j\}$ with $a_j\rightarrow0$ such that
$\inf_j\gamma_j(a_j)>0$ for each $j\in\mathbb{N}$.
\end{definition}

\begin{lemma}
\label{lem5-0}
Let $\epsilon,\delta>0$ and $(a_j)\in \ell_1$. Then there exists a sequence $(c_j)\in \ell_0$ such that
$(a_j/c_j)\in\ell_1$,
\begin{equation*}
\max_j{|c_j|}\le\frac{1}{\epsilon}\quad \text{and}\quad  \epsilon\left(1-\frac{\delta}{2}\right)\|(a_j)\|_{\ell_1}\le\left\|\left(\frac{a_j}{c_j}\right)\right\|_{\ell_1}
\le\epsilon(1+\delta)\|\{a_j\}\|_{\ell_1}.
\end{equation*}
\end{lemma}

\begin{lemma}
\label{lem5-0-1}
If a sequence of moduli of continuity $\{\gamma_j\}$ is shored-up, then $\Gamma=\cup_{j\in\mathbb{N}}\{\gamma_j\}$ is non-collapsing.
\end{lemma}

In what follows, we turn our attention to the proof of Theorem \ref{thm3}.
The compactness result below, which is similar to Lemma \ref{lem3-2}, states that solutions of a perturbed equation are H\"{o}lder continuous.

\begin{lemma}
\label{lem5-1}
Let $\sigma\in(1,2)$ and assume that $u\in C(B_1)$ is a viscosity solution of the $\xi$-perturbed equation
\begin{equation}
\label{5-1-1}
\gamma(|Du+\xi|)\mathcal{I}_\sigma(u,x)+H(x,Du+\xi)=f(x) \quad\text{in } B_1
\end{equation}
with $\xi\in\mathbb{R}^N$. Suppose that $\gamma(1)\ge1$, ($A_1$), ($A_2$) and ($A_6$) are in force. Then $u$ is locally H\"{o}lder continuous in $B_1$ with the estimate
$$
|u(x)-u(y)|\le C|x-y|^\beta  \quad\text{for } x,y\in B_{\frac12},
$$
where $\beta\in (0,1)$ and $C\ge1$ are two universal constants.
\end{lemma}

\begin{proof}
The proof is very analogous to that of Lemma \ref{lem3-2}, so we we apply the notation from the proof of Lemma \ref{lem3-2} and focus on the differences.
Corresponding to \eqref{3-2-6}, we arrive at
\begin{align*}
&\mathcal{I}_\sigma(\widetilde{v},\overline{x})-\mathcal{I}_\sigma(\widetilde{w},\overline{y}) \\
&\qquad\ge \frac{f(\overline{y})}{\gamma(|D\widetilde{\eta}(\overline{y})+\xi|)}-\frac{f(\overline{x})}{\gamma(|D\widetilde{h}(\overline{x})+\xi|)}
+\frac{H(\overline{x},D\widetilde{h}(\overline{x})+\xi)}{\gamma(|D\widetilde{h}(\overline{x})+\xi|)}-
\frac{H(\overline{y},D\widetilde{\eta}(\overline{y})+\xi)}{\gamma(|D\widetilde{\eta}(\overline{y})+\xi|)}\\
&\qquad\ge-2\|f\|_{L^\infty(B_1)}-\left(\frac{\mathcal{M}(1+\gamma(|D\widetilde{h}(\overline{x})+\xi|))}{\gamma(|D\widetilde{h}(\overline{x})+\xi|)}+
\frac{\mathcal{M}(1+\gamma(|D\widetilde{\eta}(\overline{y})+\xi|))}{\gamma(|D\widetilde{\eta}(\overline{y})+\xi|)}\right) \\
&\qquad\ge-2(\mathcal{M}+\|f\|_{L^\infty(B_1)}),
\end{align*}
where we used the facts that $\gamma(1)\ge1$ and $|D\widetilde{h}(\overline{x})+\xi|,|D\widetilde{\eta}(\overline{y})+\xi|>1$.
Furthermore, corresponding to \eqref{3-2-9}, we obtain
\begin{align*}
\widetilde{C}L\le C(\|f\|_{L^\infty(B_1)}+\|u\|_{L^\infty(B_1)}+\|u\|_{L^1_\sigma}+\mathcal{M}).
\end{align*}
Therefore, choosing $L\ge1$ sufficiently large leads to a contradiction with the inequality above.
\end{proof}

Next, we will make a tangential analysis to establish an approximation result similar to Lemma \ref{lem4-1}.

\begin{lemma}
\label{lem5-2}
Assume that the conditions $(A_1)$, $(A_2)$, $(A_4)$ and $(A_6)$ hold. Let $\gamma\in \Gamma$, where $\Gamma$ is a collection of non-collapsing moduli of continuity fulfilling $\gamma(1)\ge1$,
and assume that $u\in C(B_1)$ is a normalized viscosity solution to \eqref{5-1-1}. Given $\mu,\varepsilon>0$, there exists $\kappa>0$,
depending on $N$, $\lambda$, $\Lambda$, $\varepsilon$, $\mu$ and $\Gamma$, such that if
$$
|\sigma-2|+\|f\|_{L^\infty(B_1)}+\mathcal{M}\leq \kappa
$$
and
$$
|u(x)|\le\mu(1+|x|^{1+\alpha}), \quad  x\in\mathbb{R}^N,
$$
with some $\alpha\in(0,1)$, then there exists an $F$-harmonic function $h\in C^{1,\alpha}(B_{\frac34})$ satisfying
$$
\|u-h\|_{L^\infty(B_{\frac12})}\leq \varepsilon.
$$
\end{lemma}

\begin{proof}
For a contradiction, assume that the claim does not hold. Then there exist $\mu_0,\varepsilon_0>0$ and sequences $(\sigma_j)$, $(\gamma_j)$, $(f_j)$, $(u_j)$, $(H_j)$ and  $(\xi_j)$ such that
\begin{itemize}
  \item[(i)]
  $u_j\in C(B_1)$, with $\|u_j\|_{L^\infty(B_1)}\leq 1$ and $|u_j(x)|\le\mu_0(1+|x|^{1+\alpha})$, is a viscosity solution to
      \begin{equation*}
      -\gamma_j(|Du_j+\xi_j|)\mathcal{I}_{\sigma_j}(u_j,x)+H_j(x,Du_j+\xi_j)=f_j(x)  \quad \text{in } B_1,
      \end{equation*}
      where the operator $\mathcal{I}_{\sigma_j}$ satisfies $(A_4)$ as well;

    \smallskip

  \item[(ii)] $\gamma_j$ is a modulus of continuity with $\gamma_j(0)=0$ and $\gamma_j(1)\ge1$. Moreover, if $\gamma_j(a_j)\rightarrow0$, then $a_j\rightarrow0$;

  \item[(iii)] the Hamiltonian term fulfills
  \begin{equation*}
  |H_j(x,\xi)|\le \mathcal{M}_j(1+\gamma_j(|\xi|));
  \end{equation*}

    \smallskip

  \item[(iv)] and the smallness condition
  \begin{equation*}
|\sigma_j-2|+\|f_j\|_{L^\infty(B_1)}+\mathcal{M}_j\leq \frac{1}{j}
\end{equation*}
holds.
\end{itemize}
Moreover, we have
$$
\|u_j-h\|_{L^\infty(B_{\frac12})}>\varepsilon_0.
$$
for every $h\in C^{1,\alpha}_{\rm loc}(B_1)$.

The rest of the proof is very similar to that of Lemma \ref{lem4-1} by utilizing Lemma \ref{lem5-1}), where we replace $|Du_j+\xi_j|^p$ by $\gamma_j(|Du_j+\xi_j|)$. It is worth pointing out that when considering the term $\gamma_j(\cdot)$, we need invoke the non-collapsing property of the set $\Gamma$ (that is, $\gamma_j(a_j)\rightarrow0$ implies $a_j\rightarrow0$), or we make use of the equivalent property on non-collapsing sets in \cite[Proposition 2 (2)]{APPT22}. For fully nonlinear equations $\gamma(|Du+\xi|)F(D^2u)=f$ without Hamiltonian term, such approximation theory can be found in \cite[Proposition 6]{APPT22}.
\end{proof}

In the sequel, we shall concentrate on the existence of approximating hyperplanes. Let $h\in C^{1,\alpha}(B_{\frac12})$ is an $F$-harmonic function from Lemma \ref{lem5-2}. Let $L\ge1$ be a number satisfying $\|h\|_{C^{1,\alpha}(B_{\frac12})}\le L$. Consider two moduli of continuity
$$
\beta(t)=t\gamma(t) \quad\text{and}\quad  \omega(t)=\beta^{-1}(t).
$$

Next, we choose $0<\nu_1<1$ as follows. If $t^\alpha=o(\omega(t))$, we take $0<r<\frac12$ so small that
$$
\nu_1=\omega(r)=2Lr^\alpha>r^{\sigma-1},
$$
which could be achieved by taking $\sigma$ sufficiently close to 2 (this allows $\sigma-1>\alpha$).
If $\omega(t)=O(t^\alpha)$, we let $0<\theta <\alpha$ and choose $0<r<\frac12$ so small that
$$
\nu_1=r^\theta=2Lr^\alpha>r^{\sigma-1},
$$
where we applied $0<\theta<\alpha<\sigma-1$. Observe that once we determine $0<\theta <\alpha$, the previous choice becomes universal.

We proceed by letting
$$
0<\vartheta=\frac{r^{\sigma-1}}{\nu_1}<1
\quad\text{and}\quad
(a_k)_{k\in\mathbb{N}}=(\gamma^{-1}(\vartheta^k))_{k\in\mathbb{N}}.
$$
Because the inverse $\gamma^{-1}$ is Dini continuous, we may conclude that $(\gamma^{-1}(\vartheta^k))_{k\in\mathbb{N}}\in\ell_1$. With the help of Lemma \ref{lem5-0}, we can find a sequence $(c_k)_{k\in\mathbb{N}}\in \ell_0$ 
fulfilling
\begin{equation}
\label{5-4}
\frac{9}{10}\sum^\infty_{k=1}\gamma^{-1}(\vartheta^k)\le \sum^\infty_{k=1}\frac{\gamma^{-1}(\vartheta^k)}{c_k}\le
\sum^\infty_{k=1}\gamma^{-1}(\vartheta^k).
\end{equation}
Then we construct a sequence of moduli of continuity $(\gamma_k(t))_{k\in\mathbb{N}}$ by a recursive formula, the derivation of which is from Proposition \ref{pro5-4} below.
\begin{align*}
\gamma_0(t)&=\gamma(t),\\
\gamma_1(t)&=\frac{\nu_1}{r^{\sigma-1}}\gamma(\nu_1t),\\
\gamma_2(t)&=\frac{\nu_2\nu_1}{r^{2(\sigma-1)}}\gamma(\nu_2\nu_1t),\\
 \vdots \\
\gamma_k(t)&=\frac{\prod^k_{i=1}\nu_i}{r^{k(\sigma-1)}}\gamma\left(\prod^k_{i=1}\nu_it\right),
\end{align*}
where $\nu_1>r^{\sigma-1}$ has been defined as above, and $\nu_k$ for $k\ge2$ is determined by the following algorithm. If
$$
\frac{\nu_1^2}{r^{2(\sigma-1)}}\gamma(\nu_1^2c_2)\ge1,
$$
then we choose $\nu_2=\nu_1$; otherwise, we choose $\nu_2\in(\nu_1,1)$ so that
$$
\frac{\nu_2\nu_1}{r^{2(\sigma-1)}}\gamma(\nu_1\nu_2c_2)=1,
$$
where $c_2$ is the second element of the sequence $(c_k)_{k\in\mathbb{N}}\in \ell_0$ for which \eqref{5-4} is valid. Then we recursively apply this algorithm.
Assume that we have chosen $r^{\sigma-1}<\nu_1\le \nu_2\le\dots\le\nu_k<1$.
Then we select $\nu_{k+1}$ as follows. If
$$
\frac{\nu_k\prod^k_{i=1}\nu_i}{r^{(k+1)(\sigma-1)}}\gamma\biggl(\nu_k\biggl(\prod^k_{i=1}\nu_i\biggr)c_{k+1}\biggr)\ge1,
$$
we let $\nu_{k+1}=\nu_k$. Otherwise, we choose $\nu_k<\nu_{k+1}<1$ such that $\gamma_{k+1}(c_{k+1})=1$. Here $c_{k+1}$ is the $(k+1)$-th element of $(c_k)_{k\in\mathbb{N}}\in \ell_0$ for which \eqref{5-4} is true. According to Definition \ref{def5-2}, we know that the sequence of moduli of continuity $(\gamma_k(t))_{k\in\mathbb{N}}$ is shored-up. Let
$$
\Gamma=\{\gamma_0(t),\gamma_1(t),\dots,\gamma_k(t),\dots\}.
$$
Then, employing Lemma \ref{lem5-0-1}, the collection $\Gamma$ is non-collapsing.

\begin{lemma}
\label{lem5-3}
Assume that the conditions $(A_1), (A_2)$, $(A_4)$--$(A_6)$ hold. Let $u\in C(B_1)$ be a normalized solution to \eqref{main3}. Given $\mu>0$, there exists $\kappa>0$ such that if
$$
|\sigma-2|+\|f\|_{L^\infty(B_1)}+\mathcal{M}\leq \kappa
$$
and
$$
|u(x)|\le\mu(1+|x|^{1+\alpha}), \quad  x\in\mathbb{R}^N,
$$
there exists an affine function $l(x)=a+b\cdot x $ with $|a|+|b|\le L$ and 
$$
\|u-l\|_{L^\infty(B_r)}\leq \nu_1r,
$$
where $0<r<\frac12$ and $L>0$ are universal constants.
\end{lemma}

\begin{proof}
By means of Lemma \ref{lem5-2}, we know that there exists an $F$-harmonic function $h\in C^{1,\alpha}_{\rm loc}(B_1)$ such that
$$
\sup_{x\in B_{\frac12}}|u(x)-h(x)|\le \varepsilon
$$
for some $\varepsilon>0$ to be fixed later. Once $\varepsilon$ is chosen, we may determine the value of $\kappa$ in Lemma \ref{lem5-2}. From the regularity theory, it is well known that, for a universal constant $L>0$ and any $0<r<\frac12$, we have
$$
\sup_{x\in B_r}|h(x)-h(0)-Dh(0)\cdot x|\le Lr^{1+\alpha}
$$
with $|h(0)|+|Dh(0)|\le L$. Therefore, by letting $a=h(0)$ and $b=Dh(0)$, we get
$$
\sup_{x\in B_r}|u(x)-a-b\cdot x|\le \varepsilon+Lr^{1+\alpha}=\varepsilon+\frac{\nu_1}{2}r.
$$
Finally, taking $\varepsilon=\frac{\nu_1}{2}r$ completes the proof.
\end{proof}

In the sequel, we generalize Lemma \ref{lem5-3} to arbitrarily small radii in discrete scales, which can easily conclude the $C^1$-regularity result, Theorem \ref{thm3}.

\begin{proposition}
\label{pro5-4}
Assume that the conditions $(A_1), (A_2)$, $(A_4)$--$(A_6)$ hold. Let $u\in C(B_1)$ be a normalized solution to \eqref{main3}. Given $\mu>0$, there exists $\kappa>0$ such that if
$$
|\sigma-2|+\|f\|_{L^\infty(B_1)}+\mathcal{M}\leq \kappa
$$
and
$$
|u(x)|\le\mu(1+|x|^{1+\alpha}),\quad  x\in\mathbb{R}^N,
$$
there exists a sequence of affine functions $(l_j)$, 
$l_j(x)=A_j+B_j \cdot x$, fulfilling
$$
\sup_{x\in B_{r^j}}|u(x)-l_j(x)|\leq \biggl(\prod^j_{i=1}\nu_i\biggr)r^j,
$$
$$
|A_{j+1}-A_j|\le C\biggl(\prod^j_{i=1}\nu_i\biggr)r^j
\quad\text{and}\quad
|B_{j+1}-B_j|\le C\prod^j_{i=1}\nu_i
$$
for every $j\in\mathbb{N}$, where $0<r<\frac12$ and $C>0$ are two universal constants.
\end{proposition}

\begin{proof}
We argue by induction. Let
$$
u_1(x)=\frac{u(rx)-\overline{l}_0(rx)}{\nu_1r},
$$
with $\nu_1$ and $\overline{l}_0(x)=a_0+b_0x=l(x)$ as in Lemma \ref{lem5-3}. Then $u_1$ is a solution to
$$
\gamma_1\left(\left|Du_1+\frac{b_0}{\nu_1}\right|\right)\mathcal{I}_1(u_1,x)+H_1\left(x,Du_1+\frac{b_0}{\nu_1}\right)=f_1(x) \quad\text{in } B_1,
$$
where
$$
f_1(x)=f(rx), \quad \gamma_1(t)=\frac{\nu_1}{r^{\sigma-1}}\gamma(\nu_1t), \quad H_1(x,\xi)=H(rx,\nu_1\xi)
$$
and the nonlocal operator $\mathcal{I}_1$ possesses the same uniform ellipticity property as the $\mathcal{I}$ in \eqref{main}.
By the argument before Lemma \ref{lem5-3}, we have $\gamma_1(1)=1$. By $(A_6)$, we have
\begin{align*}
|H_1(x,\xi)|&\le \mathcal{M}(1+\gamma(\nu_1|\xi|))\\
&\le \mathcal{M}\left(1+\frac{\nu_1}{r^{\sigma-1}}\gamma(\nu_1|\xi|)\right)=\mathcal{M}(1+\gamma_1(|\xi|)).
\end{align*}
Moreover, we have $|u_1(x)|\le \mu(1+|x|^{1+\alpha})$ for $x\in\mathbb{R}^N$, which will be proved later.
Thus we have verified that $u_1$ falls into the framework of Lemma \ref{lem5-3}. Hence there exists an affine function $\overline{l}_1(x)=a_1+b_1\cdot x$ with $|a_1|+|b_1|\le L$ such that
$$
\sup_{x\in B_{r}}|u_1(x)-\overline{l}_1(x)|\leq \nu_1r.
$$

Next, let
$$
u_2(x)=\frac{u_1(rx)-\overline{l}_1(rx)}{\nu_2r}
$$
for $\nu_2\ge\nu_1>r^{\sigma-1}$ chosen earlier. Then $u_2$ is a solution to
$$
\gamma_2\left(\left|Du_2+\frac{b_1}{\nu_2}+\frac{b_0}{\nu_2\nu_1}\right|\right)\mathcal{I}_2(u_2,x)+H_2\left(x,Du_2+\frac{b_1}{\nu_2}
+\frac{b_0}{\nu_2\nu_1}\right)=f_2(x) \quad\text{in } B_1,
$$
where
$$
f_2(x)=f_1(rx)=f(r^2x), \quad \gamma_2(t)=\frac{\nu_2}{r^{\sigma-1}}\gamma_1(\nu_2t)=\frac{\nu_2\nu_1}{r^{2(\sigma-1)}}\gamma(\nu_2\nu_1t)
$$
and
$$
 H_2(x,\xi)=H_1(rx,\nu_2\xi)=H(r^2x,\nu_2\nu_1\xi),
$$
and the nonlocal operator $\mathcal{I}_2$ has the same uniform ellipticity condition as the $\mathcal{I}$ in \eqref{main}. In view of $(A_6)$, we have
\begin{align*}
|H_2(x,\xi)|&\le \mathcal{M}(1+\gamma(\nu_2\nu_1|\xi|))\\
&\le\mathcal{M}\left(1+\frac{\nu_2\nu_1}{r^{2(\sigma-1)}}\gamma(\nu_2\nu_1|\xi|)\right)=\mathcal{M}(1+\gamma_2(|\xi|)).
\end{align*}
Moreover, we have $|u_2(x)|\le \mu(1+|x|^{1+\alpha})$ for $x\in\mathbb{R}^N$, which will be proved later. That is, $u_2$ also falls into the framework of Lemma \ref{lem5-3}.
Hence there exists an affine function  $\overline{l}_2(x)=a_2+b_2\cdot x$ with $|a_2|+|b_2|\le L$ such that
$$
\sup_{x\in B_{r}}|u_2(x)-\overline{l}_2(x)|\leq \nu_1r.
$$

Recursively, let
$$
u_j(x)=\frac{u_{j-1}(rx)-\overline{l}_{j-1}(rx)}{\nu_jr}
$$
for $\nu_j\ge\nu_{j-1}\ge\dots\ge\nu_1>r^{\sigma-1}$ selected earlier. Then $u_j$ is a solution to
\begin{align*}
&\gamma_j\left(\left|Du_j+\frac{b_{j-1}}{\nu_j}+\frac{b_{j-2}}{\nu_j\nu_{j-1}}+\dots+\frac{b_0}{\nu_j\nu_{j-1}\dots\nu_1}\right|\right)
\mathcal{I}_j(u_j,x)\\
&\qquad+H_j\left(x,Du_j+\frac{b_{j-1}}{\nu_j}+\frac{b_{j-2}}{\nu_j\nu_{j-1}}+\dots+\frac{b_0}{\nu_j\nu_{j-1}\dots\nu_1}\right)=f_j(x) \quad\text{in } B_1,
\end{align*}
where $f_j(x)=f_{j-1}(rx)=\dots=f(r^jx)$,
$$
\gamma_j(t)=\frac{\nu_j}{r^{\sigma-1}}\gamma_{j-1}(\nu_jt)=\dots=\frac{\prod^j_{i=1}\nu_i}{r^{j(\sigma-1)}}\gamma\biggl(\prod^j_{i=1}\nu_it\biggr)
$$
and
$$
 H_j(x,\xi)=H_{j-1}(rx,\nu_j\xi)=\dots=H\biggl(r^jx,\biggl(\prod^j_{i=1}\nu_i\biggr)\xi\biggr)
$$
and the nonlocal operator $\mathcal{I}_j$ carries the same uniform ellipticity condition as $\mathcal{I}$ in \eqref{main}. Recall that $\nu_j\ge\nu_{j-1}$ is fixed in such way that either $\nu_j=\nu_{j-1}$ or else $\gamma_j(c_j)=1$. Thanks to $(A_6)$, we get
\begin{align*}
|H_j(x,\xi)|&\le \mathcal{M}\biggl(1+\gamma\biggl(\prod^j_{i=1}\nu_i|\xi|\biggr)\biggr)\\
&\le\mathcal{M}\biggl(1+\frac{\prod^j_{i=1}\nu_i}{r^{j(\sigma-1)}}\gamma\biggl(\prod^j_{i=1}\nu_i|\xi|\biggr)\biggr)
=\mathcal{M}(1+\gamma_j(|\xi|)).
\end{align*}

Finally, we prove that
\begin{equation}
\label{5-4-1}
|u_j(x)|\le \mu(1+|x|^{1+\alpha}), \quad x\in\mathbb{R}^N.
\end{equation}
At this stage, $u_j$ meets the requirements of Lemma \ref{lem5-3}. Then we apply again Lemma \ref{lem5-3} to find an affine function  $\overline{l}_j(x)=a_j+b_j\cdot x$ with $|a_j|+|b_j|\le L$ such that
\begin{equation}
\label{5-4-2}
\sup_{x\in B_{r}}|u_j(x)-\overline{l}_j(x)|\leq \nu_1r.
\end{equation}

Then let us justify the assertion \eqref{5-4-1} by an induction argument. For $j=0$, we take $u_0=u$. Assume \eqref{5-4-1} is true for $i=0,1,2,\dots,j-1$.
We shall prove the corresponding claim for the case $i=j$. We next distinguish two mutual exclusive situations. If $r|x|\ge\frac{1}{2}$, through the induction assumption and $2Lr^\alpha=\nu_1\le\nu_j<1$, we obtain
\begin{align*}
|u_j(x)|&\le (\nu_jr)^{-1}(|u_{j-1}(rx)|+|l_{j-1}(rx)|)\\
&\le(2Lr^{1+\alpha})^{-1}[\mu(1+|rx|^{1+\alpha})+L(1+|rx|)]\\
&\le\left(\frac{\mu}{2L}+\frac{2^\alpha\mu}{L}+2^\alpha+2^{\alpha-1}\right)|x|^{1+\alpha}=\widetilde{\mu}|x|^{1+\alpha}.
\end{align*}
We verify \eqref{5-4-1} by choosing $\mu$ such that $\widetilde{\mu}\le\mu$. On the other hand, if $r|x|<\frac{1}{2}$, exploiting again $2Lr^\alpha=\nu_1\le\nu_j<1$, we have
\begin{align*}
|u_j(x)|&\le (\nu_jr)^{-1}(|u_{j-1}(rx)-h(rx)|+|h(rx)-l_{j-1}(rx)|)\\
&\le(\nu_jr)^{-1}\left(\frac{\nu_1r}{2}+Lr^{1+\alpha}|x|^{1+\alpha}\right)
\le\frac{1}{2}+\frac{1}{2}|x|^{1+\alpha},
\end{align*}
where $h\in C^{1,\alpha}_{\rm loc}(B_1)$ is an $F$-harmonic function from Lemma \ref{lem5-2}.
This completes the proof of \eqref{5-4-1}.

By \eqref{5-4-2}, rescaling back to $u$, we obtain
$$
\sup_{x\in B_{r^2}}|u_{j-1}(x)-\overline{l}_{j-1}(x)-\nu_jr\overline{l}_j(r^{-1}x)|\le \nu_1\nu_jr^2.
$$
It follows that
$$
\sup_{x\in B_{r^3}}|u_{j-2}(x)-\overline{l}_{j-2}(x)-\nu_{j-1}r\overline{l}_{j-1}(r^{-1}x)-\nu_{j-1}\nu_jr^2\overline{l}_j(r^{-2}x)|\le \nu_1\nu_{j-1}\nu_jr^3
$$
and, recursively,
$$
\sup_{x\in B_{r^{j+1}}}|u(x)-l_{j+1}(x)|\le \nu_1^2\nu_2\nu_3\cdots\nu_jr^{j+1}\le \left(\prod^{j+1}_{i=1}\nu_i\right)r^{j+1}.
$$
Here
$$
l_{j+1}(x)=\overline{l}_0(x)+\sum^j_{i=1}\overline{l}_i(r^{-i}x)\biggl(\prod^{i}_{k=1}\nu_k\biggr)r^i=A_{j+1}+B_{j+1}\cdot x
$$
with $\overline{l}_i(x)=A_i+B_i\cdot x$. Obviously, it holds that
$$
|A_{j+1}-A_{j}|\le C\left(\prod^{j}_{k=1}\nu_k\right)r^j
\quad\text{and}\quad
|B_{j+1}-B_{j}|\le C\prod^{j}_{k=1}\nu_k.
$$
This completes the proof.
\end{proof}

Making use of Proposition \ref{pro5-4}, we could deduce the $C^1$-regularity for the viscosity solutions to \eqref{main3}. Owing to the proof of Theorem \ref{thm3} is analogous to that of \cite[Theorem 1]{APPT22}, we just sketch it here.

\begin{proof}[\textbf{Proof of Theorem \ref{thm3}}] To prove Theorem \ref{thm3}, we first reduce the problem to a smallness regime in Proposition \ref{pro5-4},  as in the proof of Theorem \ref{thm1}.
Let
$$
v(x)=\frac{u(rx)}{K}
$$
with $0<r\le1$ and $K\ge1$ to be chosen later. If $u$ is a solvution to \eqref{main3}, then $v$ is a solution to
$$
-\overline{\gamma}(|Dv|)\overline{\mathcal{I}}(v,x)+\overline{H}(x,Dv)=\overline{f}(x) \quad\text{in }  B_1,
$$
where $\overline{\mathcal{I}}$ has the same uniform ellipticity property as the nonlocal operator $\mathcal{I}$ in \eqref{main}, and
$$
\overline{\gamma}(|Dv|)=\frac{\gamma\left(\frac{K}{r}|Dv|\right)}{\gamma\left(\frac{K}{r}\right)}, \quad  \overline{f}(x)=\frac{\frac{r^\sigma}{K}}{\gamma\left(\frac{K}{r}\right)}f(rx)
$$
and
$$
\overline{H}(x,Dv)=\frac{\frac{r^\sigma}{K}}{\gamma\left(\frac{K}{r}\right)}H\left(rx,\frac{K}{r}Dv\right).
$$
Since$\gamma(1)\ge1$, we may conclude that
\begin{align*}
|\overline{H}(x,\xi)|&\le \frac{r^\sigma}{K}\gamma^{-1}\left(\frac{K}{r}\right)\mathcal{M}\left(1+\gamma\left(\frac{K}{r}|\xi|\right)\right)\\
&\le\frac{r^\sigma}{K}\mathcal{M}(1+\overline{\gamma}(|\xi|))
=\overline{\mathcal{M}}(1+\overline{\gamma}(|\xi|)).
\end{align*}
By choosing suitable numbers $r$ and $K$, we will fall within a smallness regime.

By a similar argument as in the construction of the sequence $(\nu_j)$, we conclude that $(A_j)$ and $(B_j)$ are Cauchy sequences.
Thus there exist $A_{\infty}\in\mathbb{R}$ and $B_\infty\in\mathbb{R}^N$ such that $A_j\rightarrow A_{\infty}$ and $B_j\rightarrow B_{\infty}$. Let
$$
l_\infty=A_{\infty}+B_{\infty}\cdot x.
$$
For all $0<\rho\ll1$, there exists $j\in\mathbb{N}$ satisfying $r^{j+1}<\rho\le r^j$. Then we arrive at
$$
\sup_{x\in B_\rho}|u(x)-l_\infty(x)|\leq C\omega(\rho)\rho,
$$
where $\omega(t)$ is a modulus of continuity. The details can be found in \cite[Proof of Theorem 1, Pages 31-32]{APPT22}, or in \cite[Proof of Theorem 1]{WJ25}.
\end{proof}

\section*{Acknowledgements}
This work was supported by the National Natural Science Foundation of China (Nos. 12071098, 11871134) and the Young talents sponsorship program of Heilongjiang Province (No. 2023QNTJ004).

\section*{Declarations}
\subsection*{Conflict of interest} The authors declare that there is no conflict of interest. We also declare that this
manuscript has no associated data.

\subsection*{Data availability} Data sharing is not applicable to this article as no datasets were generated or analysed
during the current study.


\begin{thebibliography}{[a]}

\bibitem{AN25} P. Andrade, T. Nascimento, Optimal regularity for degenerate elliptic equations with Hamiltonian terms, arXiv:2508.03924.

\bibitem{APPT22} P. Andrade, D. Pellegrino, E. Pimentel, E. Teixeira, $C^1$-regularity for degenerate diffusion equations, Adv. Math. 409 (2022) 34 pp.

\bibitem{APS24} P. Andrade, D. dos Prazeres, M. Santos, Regularity estimates for fully nonlinear integro-differential equations with nonhomogeneous degeneracy, Nonlinearity 37 (4) (2024) 29 pp.

\bibitem{APT23} D. Ara\'{u}jo, D. dos Prazeres, E. Topp, On fractional quasilinear equations with elliptic degeneracy, arXiv:2306.15452, 2023.

\bibitem{ART15} D. Ara\'{u}jo, G. Ricarte, E. Teixeira, Geometric gradient estimates for solutions to degenerate elliptic equations, Calc. Var. Partial Differential Equations 53 (2015) 605--625.
        

\bibitem{BBLL} S. Baasandorj, S. Byun, K. Lee, S. Lee, Global regularity results for a class of singular/degenerate fully nonlinear elliptic equations, Math. Z. 306 (1) (2024) 26 pp.


\bibitem{BCCI12} G. Barles, E. Chasseigne, A. Ciomaga, C. Imbert, Lipschitz regularity of solutions for mixed integro-differential equations, J. Differential Equations 252 (2012) 6012--6060.

\bibitem{BCI11} G. Barles, E. Chasseigne, C. Imbert, H\"{o}lder continuity of solutions of second-order non-linear elliptic integro differential equations, J. Eur. Math. Soc. 13 (1) (2011) 1--26.

\bibitem{BKLT15} G. Barles, S. Koike, O. Ley, E. Topp, Regularity results and large time behavior for integro-differential equations with coercive Hamiltonians, Calc. Var. Partial Differential Equations 54 (1) (2015) 539--572.

\bibitem{BT16} G. Barles, E. Topp, Lipschitz regularity for censored subdiffusive integro-differential equations with superfractional gradient terms, Nonlinear Anal. 131 (2016) 3--31.

\bibitem{BD16} I. Birindelli, F. Demengel, Fully nonlinear operators with Hamiltonian: H\"{o}lder regularity of the gradient, NoDEA Nonlinear Differential Equations Appl. 23 (4) (2016) 17 pp.

\bibitem{BD10} I. Birindelli, F. Demengel, Regularity and uniqueness of the first eigenfunction for singular, fully nonlinear elliptic operators, J. Differential Equations 249 (2010) 1089--1110.

\bibitem{BDL19} I. Birindelli, F. Demengel, F. Leoni, $C^{1,\gamma}$ regularity for singular or degenerate fully nonlinear equations and applications, NoDEA Nonlinear Differential Equations Appl. 26 (5) (2019) 13 pp.

\bibitem{BK23} A. Biswas, S. Khan, Existence-uniqueness for nonlinear integro-differential equations with drift in $\mathbb{R}^d$, SIAM J. Math. Anal. 55 (5) (2023) 4378--4409.

\bibitem{BQT25} A. Biswas, A. Quaas, E. Topp, Nonlocal Liouville theorems with gradient nonlinearity, J. Funct. Anal. 289 (8) (2025) 44 pp.

\bibitem{BT25a} A. Biswas, E. Topp, Nonlocal ergodic control problem in $\mathbb{R}^d$, Math. Ann. 390 (1) (2024) 45--94.

\bibitem{BT25} A. Biswas, E. Topp, Lipschitz Regularity of Fractional p-Laplacian, Ann. PDE 11 (2) (2025) Paper No. 27.

\bibitem{BKO25} S-S Byun, H. Kim, J. Oh, Interior $W^{2,\delta}$ type estimates for degenerate fully nonlinear elliptic equations with $L^n$ data, J. Funct. Anal. 289 (6) (2025) 37~pp.

\bibitem{CC95} L. Caffarelli, X. Cabr\'{e}, Fully Nonlinear Elliptic Equations, American Mathematical Society Colloquium Publications, vol. 43, American Mathematical Society, Providence, RI, 1995.

\bibitem{CS09} L. Caffarelli, L. Silvestre, Regularity theory for fully nonlinear integro-differential equations, Commun. Pure Appl.
Math. 62 (5) (2009) 597--638.

\bibitem{CS11} L. Caffarelli, L. Silvestre, Regularity results for nonlocal equations by approximation, Arch. Ration. Mech. Anal.
200 (1) (2011) 59--88.

\bibitem{CR11} P. Cardaliaguet, C. Rainer, H\"{o}lder regularity for viscosity solutions of fully nonlinear, local or nonlocal, Hamilton-Jacobi equations with superquadratic growth in the gradient, SIAM J. Control Optim. 49 (2) (2011) 555--573.

\bibitem{CD12} H. Chang Lara, G. D\'{a}vila, Regularity for solutions of nonlocal, nonsymmetric equations, Ann. Inst. H. Poincar\'{e} C Anal. Non Lin\'{e}aire 29 (6) (2012) 833--859.

\bibitem{CLN19} E. Chasseigne, O. Ley, T. Nguyen, A priori Lipschitz estimates for solutions of local and nonlocal Hamilton-Jacobi equations with Ornstein-Uhlenbeck operator, Rev. Mat. Iberoam. 35 (5) (2019) 1415--1449.

\bibitem{DQT24} G. D\'{a}vila, A. Quaas, E. Topp, On large solutions for fractional Hamilton-Jacobi equations, Proc. Roy. Soc. Edinburgh Sect. A, 154 (5) (2024) 1313--1335.
    

\bibitem{DeF} C. De Filippis, Regularity for solutions of fully nonlinear elliptic equations with nonhomogeneous degeneracy, Proc. Roy. Soc. Edinburgh Sect. A 151 (2021) 110--132.


\bibitem{DJZ18} H. Dong, T. Jin, H. Zhang, Dini and Schauder estimates for nonlocal fully nonlinear parabolic equations with drifts, Anal. PDE 11 (6) (2018) 1487--1534.

\bibitem{FRZ21} Y. Fang, V. R\u{a}dulescu, C. Zhang, Regularity of solutions to degenerate fully nonlinear elliptic equations with variable exponent, Bull. Lond. Math. Soc. 53 (6) (2021) 1863--1878.

\bibitem{FRZ25} Y. Fang, V. R\u{a}dulescu, C. Zhang, Regularity for a class of degenerate fully nonlinear nonlocal elliptic equations, Calc. Var. Partial Differential Equations 64 (5) (2025) 29~pp.

\bibitem{Imb11} C. Imbert, Alexandroff-Bakelman-Pucci estimate and Harnack inequality for degenerate/singular fully non-linear elliptic equations, J. Differential Equations 250 (2011) 1553--1574.

\bibitem{IS13} C. Imbert, L. Silvestre, $C^{1,\alpha}$ regularity of solutions of some degenerate fully non-linear elliptic equations, Adv. Math. 233 (2013) 196--206.

\bibitem{JLMS} E. J\'{u}nior, J. da Silva, G. Rampasso, G. Ricarte, Global regularity for a class of fully nonlinear PDEs with unbalanced variable degeneracy, J. Lond. Math. Soc. (2) 108 (2) (2023) 622--665.

\bibitem{Nas24} T. Nascimento, Schauder-type estimates for fully nonlinear degenerate elliptic equations, J. Funct. Anal. 289 (1) (2025) 23~pp.

\bibitem{OT23} P. Oza, J. Tyagi, Regularity of solutions to variable-exponent degenerate mixed fully nonlinear local and nonlocal equations, arXiv:2302.06046v1.

\bibitem{PT21} D. dos Prazeres, E. Topp, Interior regularity results for fractional elliptic equations that degenerate with the gradient, J. Differential Equations 300 (2021) 814--829.

\bibitem{QSX20} A. Quaas, A. Salort, A. Xia, Principal eigenvalues of fully nonlinear integro-differential elliptic equations with a drift term, ESAIM Control Optim. Calc. Var. 26 (2020) 19 pp.

\bibitem{Soner} H. Soner, Optimal control with state-space constraint II, SIAM J. Control Optim. 24 (6) (1986) 1110--1122.

\bibitem{WJ25} J. Wang, F. Jiang, Regularity of solutions for degenerate or singular fully nonlinear integro-differential equations, Commun. Contemp. Math., https://doi.org/10.1142/S0219199725500804.
\end{thebibliography}
\end{document}